\documentclass[11pt]{article}
\usepackage{amssymb,amsbsy,amsmath,amsfonts,amssymb,amscd,amsthm}
\usepackage{fullpage}

\numberwithin{equation}{section}
\usepackage[T1]{fontenc}
\usepackage[toc,page]{appendix}
\usepackage{hyperref}
\usepackage{url}
\usepackage{courier}
\usepackage{easyReview}
\usepackage[capitalise]{cleveref}
\usepackage{mathtools}
\usepackage{xcolor}
\graphicspath{{./}{figures/}} 
\usepackage{tikz}
\usepackage[linesnumbered,ruled,lined,noend]{algorithm2e}
\SetKwRepeat{Do}{do}{while}
\usetikzlibrary{matrix}
\usepackage{algpseudocode}

\newtheorem{theorem}{Theorem}[section]
\newtheorem{lemma}[theorem]{Lemma}
\newtheorem{proposition}[theorem]{Proposition}
\newtheorem{corollary}[theorem]{Corollary}
\newtheorem{remark}[theorem]{Remark}

\newtheorem{definition}[theorem]{Definition}
\newtheoremstyle{Claim}{}{}{\itshape}{}{\itshape\bfseries}{:}{ }{#1}
\theoremstyle{Claim}

\newcommand{\NN}{{\mathbb{N}}}

\newcommand{\EE}{\mathbb{E}}
\newcommand{\PP}{\mathbb{P}}
\newcommand{\R}{\mathbb{R}}

\newcommand{\N}{\mathbb{N}}

\newcommand{\mA}{\mathcal A}
\newcommand{\mC}{\mathcal C}

\def\bq{{\mathsf{q}}}
\def\bm{{\mathsf{m}}}
\def\sU{{\mathsf{U}}}
\def\sV{{\mathsf{V}}}
\def\su{{\mathsf{u}}}
\def\sv{{\mathsf{v}}}

\def\i{{(\iota)}}
\def\bfs{{\mathbf{s}}}
\def\bV{{\mathbf{V}}}
\def\bG{{\mathbf{G}}}

\def\bjm{{\bar{\jmath}}}

\newcommand{\pd}{\partial}
\newcommand{\ux}{{\underline x}}

\newcommand{\bc}{\mathbf{c}}
\newcommand{\bU}{\mathbf{U}}

\newcommand{\beq}{\begin{equation}}
\newcommand{\eeq}{\end{equation}}

\newcommand{\e}{\epsilon}

\theoremstyle{plain}

\def\sideremark#1{\ifvmode\leavevmode\fi\vadjust{
		\vbox to0pt{\hbox to 0pt{\hskip\hsize\hskip1em
				\vbox{\hsize3cm\tiny\raggedright\pretolerance10000
					\noindent #1\hfill}\hss}\vbox to8pt{\vfil}\vss}}}

\newcommand{\D}{\Delta}
\newcommand{\Dx}{{\Delta x}}

\titlepage
\title{A semi-Lagrangian method for solving state constraint Mean Field Games in Macroeconomics}
\author{Fabio Camilli\thanks{Department of Engineering and Geology, University of  Chieti-Pescara ``G. d'Annunzio''}, Qing Tang\thanks{School of Mathematics and Physics, China University of Geosciences (Wuhan)}, Yong-shen Zhou\thanks{Department of Basic and Applied Sciences for Engineering, University of Rome La Sapienza}}
\date{}

\begin{document}
\maketitle

\begin{abstract} 
We study continuous-time heterogeneous agent models cast as Mean Field Games, in the Aiyagari-Bewley-Huggett framework. The model couples a Hamilton–Jacobi–Bellman   equation for individual optimization with a Fokker–Planck–Kolmogorov   equation for the wealth distribution. We establish a comparison principle for constrained viscosity solutions of the HJB equation and propose a semi-Lagrangian (SL) scheme for its numerical solution, proving convergence via the Barles–Souganidis method. A policy iteration algorithm handles state constraints, and a dual SL scheme is used for the FPK equation. Numerical methods are presented in a fully discrete, implementable form.
\end{abstract}

\section{Introduction}
This paper studies the continuous-time heterogeneous agent models developed by Achdou, Han, Lasry, Lions, and Moll \cite{achdou2022income}, which recast the classical Aiyagari-Bewley-Huggett models \cite{huggett1993risk,aiyagari1994uninsured,ljungqvist2018recursive} in the language of Mean Field Games (MFG). These models describe economies with a continuum of agents who are ex ante identical but become ex post heterogeneous due to idiosyncratic income shocks and borrowing constraints. 
The methodology developed in \cite{achdou2022income} for continuous-time heterogeneous agent models has been widely applied in the macroeconomic literature to study topics such as wealth distribution, income inequality, and fiscal policy \cite{ahn2018inequality,kaplan2018monetary,bhandari2021inequality,fernandez2023financial}. A finite difference scheme for numerically solving the associated PDE systems was proposed in \cite{achdou2022income}, while \cite{bilal2023solving} introduced a perturbation approach to solve a master equation in this context. More recently, deep learning techniques have been employed for high-dimensional settings in \cite{achdou2022simulating,gu2024global}. \par
In the MFG approach, each agent maximizes their inter-temporal utility while taking the interest rate as given. This leads to a system of coupled partial differential equations describing both individual behavior and the evolution of the population distribution. Concretely, each agent solves the following stochastic control problem, given an interest rate $r$:
\begin{equation}\label{control_problem}
	\mathbb{E}\left[\int_0^{\infty} e^{-\rho t} u(c_t) \, dt\right], \quad \text{subject to} \quad
\begin{cases}
	dx_t = (r x_t + y_t - c_t)dt, \\
	x_t \geq \underline{x},
\end{cases}
\end{equation}
where $x_t$ denotes wealth, $c_t$ consumption, $y_t \in \{y_1, y_2\}$ is the income modeled by a Poisson process, $\rho$ is the discount rate and the utility function  $u$ is  strictly increasing and concave.
The optimal control problem leads to a weakly coupled  system of Hamilton-Jacobi-Bellman (HJB) equations for the value functions $v_j(x)$, $j \in \{1, 2\}$, of agents in income state $y_j$:
$$
\rho v_j(x) = \sup_{c \ge 0} \left\{ u(c) + (r x + y_j - c) Dv_j \right\} + \lambda_j (v_{\bar{\jmath}}(x) - v_j(x)), \quad \bar{\jmath} = 3 - j,
$$
subject to a state constraint at $\underline{x}$. 
In the spirit of Bewley-type models, the equilibrium interest rate $r$ is not fixed exogenously but determined endogenously based on aggregate variables such as capital and productivity. The wealth-income distribution of agents is given by a measure $dm$ on $[\ux,\infty)\times \{y_1,y_2\}$ of the form $dm=\sum_{j\in\{1,2\}}dm_{j}(x)\otimes\delta_{y_j}(y)$ where $dm_{j}(x)$ is a measure on $[\ux,\infty)$. We assume $m_j$ is the sum of  an absolutely continuous part with density $g_j$ and possibly a Dirac mass at $x=\ux$: for each set $\Omega= [\ux,R]$ with $R<+\infty$, 
\beq\label{eq:mg}
m_j(\Omega)=\int_{\Omega}g_j(x)dx+\mu_j\delta_{\Omega}(\ux).
\eeq

The density $g_j$, in the region $x>\ux$, is described by the solution to a stationary Fokker-Planck-Kolmogorov (FPK) equation in the sense of distributions:
$$
- \frac{d}{dx} \left[ (rx + y_j - c_j(x)) g_j(x) \right] + \lambda_{\bar \jmath}g_{\bar \jmath}(x) - \lambda_j g_j(x) = 0,\quad \sum_j\int_{x>\ux}g_j(x)dx+\mu_j=1.
$$

The total (aggregate) capital and labor in the economy are given by:
\begin{equation}\label{def K N}
	K[m] = \sum_{j=1}^{2} \int_{x > \underline{x}} x g_j(x)dx+\mu_j\ux , \quad N[m] = \sum_{j=1}^{2} \int_{x> \underline{x}} y_j g_j(x) \, dx+y_j\mu_j.
\end{equation}
We now describe two classical closures of this system: the Huggett and Aiyagari models.\par

\textbf{Huggett Model:} In the Huggett framework \cite{huggett1993risk}, agents can borrow or lend at the interest rate $r$, and total net borrowing in the economy is fixed at a value $B$. The stationary recursive equilibrium is summarized by the following MFG system:
\begin{equation}\label{MFG}
	\left\{
\begin{aligned}
	&(i) \quad && \rho v_j(x) = \sup_{c \ge 0} \left\{ u(c) + (r x + y_j - c) Dv_j(x) \right\} + \lambda_j (v_{\bar{\jmath}}(x) - v_j(x)), \\
	&&& c_j^*(x) = \arg\max_{c \ge 0} \left\{ u(c) + (r x + y_j - c) Dv_j(x) \right\}, \\
	&(ii) \quad && - \frac{d}{dx} \left[ (r x + y_j - c_j^*(x)) g_j(x) \right] + \lambda_{\bar \jmath} g_{\bar \jmath}(x) - \lambda_j g_j(x) = 0, \\
	&&& \sum_j\int_{x>\ux}g_j(x)dx+\mu_j=1,
\end{aligned}
\right.
\end{equation}
together with the equilibrium condition:
\begin{equation}\label{Huggett}
	(iii_H) \quad K[m] = B.
\end{equation}
\par
\textbf{Aiyagari Model:}
The Aiyagari model interprets individual assets $x$ as holdings of physical capital. The aggregate economy is characterized by a Cobb-Douglas production function with total factor productivity $A$, capital depreciation rate $\delta$, and capital share $\alpha \in (0,1)$:
$
F(K, N) = A K^\alpha N^{1 - \alpha}.
$
The interest rate $r$ is determined from the marginal productivity of capital:
\begin{equation}\label{Aiyagari}	
(iii_A) \quad r = A \alpha \left( \frac{K[m]}{N[m]} \right)^{\alpha - 1} - \delta, 
\end{equation}
which is derived from the first-order condition $\partial_K F = r + \delta$. In fact, the $N[m]$ in the stationary Aiyagari model with \eqref{Aiyagari} has an explicit expression. From \cite[Eq. (32), p. 64]{achdou2022income} the measure $dm_j$ satisfies 
\beq\label{Ng}
\int_{x\geq \ux}dm_j=\int_{x>\ux}g_j(x)dx+\mu_j= \frac{\lambda_{\bar \jmath}}{\lambda_j + \lambda_{\bar \jmath}}, \quad \text{hence} \quad
N[m] = \frac{y_{\bar \jmath} \lambda_j + y_j \lambda_{\bar \jmath}}{\lambda_j + \lambda_{\bar \jmath}}.
\eeq
In \cite{achdou2022income}, the notion of a  constrained viscosity solution  was proposed as the appropriate framework for the HJB equation arising in heterogeneous agent models. This notion has been studied in depth in the literature (see \cite{alvarez1997convex, bardi1997optimal, capuzzo1990hamilton, Soner_1986}), with related developments for weakly coupled monotone systems in, e.g., \cite{ishii1991viscosity}. However, these results do not directly apply to our setting. Specifically, when the utility function $u$ has a CRRA form, the associated Hamiltonian \eqref{def H} is finite only for $p \geq 0$, an assumption that lies outside the standard framework in the viscosity solution literature.

As a first theoretical contribution, we establish a strong comparison principle between upper semi-continuous subsolutions and lower semi-continuous supersolutions, adapting Barles' result \cite[Theorem 4.6]{barles1994solutions}. In addition, we prove that the optimal consumption and saving policies vary continuously with respect to perturbations in the interest rate $r$.

To approximate solutions to the HJB equation, we develop a semi-Lagrangian (SL) scheme tailored to Aiyagari–Bewley–Huggett-type models. These schemes exploit the dynamic programming principle, computing the value function by tracing characteristics backward in time. SL methods are well-established in the optimal control literature (see \cite{achdou2008homogenization, alla2015efficient, bardi1997optimal,briani2012approximation,calzola2023semi,carlini2024lagrange, Chen_2008,falcone2013semi,festa2018domain}) and have been widely adopted for Mean Field Games (e.g., \cite{ashrafyan2025fully,CALZOLA2024115769, carlini2014fully, MR3392626, chowdhury2023numerical}). In addition to approximating the value function, SL schemes can also be used to construct approximate feedback (closed-loop) optimal controls. Using the strong comparison principle for the continuous problem, we prove the convergence of the numerical scheme to the constrained viscosity solution of the HJB equation. This is achieved through the classical method of relaxed limits within the Barles–Souganidis framework.

For the infinite-horizon setting, we employ a policy iteration method (Howard's algorithm) combined with the SL scheme. Howard algorithm with SL schemes has been considered in \cite{alla2015efficient,festa2018domain}. This approach is particularly effective, as state constraints can be explicitly handled during the policy update step (cf. \cite[Eq. (3.2)]{achdou2022simulating}). To solve the FPK equation, we adopt a dual SL scheme in line with recent work in the MFG literature (e.g., \cite{ashrafyan2025fully, carlini2014fully, chowdhury2023numerical}). In practice, this allows the use of a ``matrix transposition" trick, analogous to that used in finite difference schemes. Even though the convergence analysis is not addressed in detail for the FPK equation, the scheme is consistent with the weak solution formulation. Finally, we present the numerical algorithms in a fully discrete vectorized form to illustrate the implementation.\par

We also consider, only at a numerical level,  the evolutive  MFG  system describing the transition dynamics in the Aiyagari model:
\begin{equation}\label{Dynamic MFG}
	\left\{\begin{aligned}
		&(i)&&\rho v_j(t,x) = \frac{\pd v}{\pd t}+\sup_{c\ge 0} \left\{ u(c) +(r(t)x + y_j - c)Dv_j(t,x) \right\}\\
		&&&\quad +\lambda_j(v_{\bar \jmath}(t,x)-v_j(t,x)),\quad v_j(T,x)=v^{{st}}_j(x),\\
		&&&c_{j}^*(t,x)=\mathop{\arg \max}\limits_{c\ge 0} \left\{ u(c) +(r(t) x + y_j - c)Dv_j(t,x) \right\},\\
		&(ii)&&\frac{\pd g}{\pd t}+ \frac{\partial}{\partial x} \left[ \left(r(t) x + y_j - c_{j}^*(t,x)\right) g_j(t,x) \right]=\lambda_{\bar \jmath}g_{\bar \jmath}(t,x) -\lambda_jg_j(t,x),\\
		&&& g_j(x,0)=\mathsf{g}_j(x),\,\sum_j\int_{x>\ux}\mathsf{g}_j(x)dx+\mu_j(0)=1,
	\end{aligned}\right.
\end{equation}
coupled with the condition
	\begin{equation}\label{Dynamic Aiyagari}
		(iii_A)\qquad r(t) = A(t)\alpha \left( \frac{K[m(t)]}{N[m(t)]} \right)^{\alpha - 1} - \delta.
	\end{equation}	
	Here, the system has as initial condition $m_j(0)$ which is the sum of  an absolutely continuous part with density $\mathsf{g}_j(x)$ and possibly a Dirac mass at $x=\ux$ with weights $\mu_j(0)$ while $K[m(t)]$ and $N[m(t)]$ are defined as in \eqref{def K N}. The system \eqref{Dynamic MFG}–\eqref{Dynamic Aiyagari} has a typical backward-forward MFG structure.
	We assume that the terminal condition $v^{st}(x) = (v_1^{st}(x), v_2^{st}(x))$ is given by the solution of the stationary Aiyagari model \eqref{MFG}–\eqref{Aiyagari}. \par


	The paper is organized as follows. In Section 2, we establish the strong comparison principle for the HJB equation with some fixed $r$, and some properties of solution to the HJB equation. In Section 3, we give the SL scheme for the HJB equation and dual SL scheme for the FPK equation. We prove the convergence of the scheme for the HJB equation. We then discuss the implementation of algorithms with SL schemes for the MFG systems. Finally, we present some numerical results. 
\section{Constrained viscosity solution: theoretical aspects}
In this section, we discuss some theoretical aspects of HJB equation. We first make some standing parameter assumptions throughout the paper:
\beq\label{standing assump}
\begin{aligned}
&(i)\quad\text{The discount}\,\rho: \quad \rho>0.\\
&(ii)\quad\text{The interest rate}\,r:\quad -\infty<r<\rho<+\infty.\\
&(iii)\quad\text{The income}\,y_j:\quad 0<y_1<y_2.\\
&(iv)\quad\text{Risk aversion}\,\gamma:\quad \gamma>1.\\
&(v)\quad\text{Total credit supply}\,B:\quad B\geq 0\quad \text{and}\quad B>\ux.\\
&(vi)\quad\text{State constraint}\,\ux:\quad \ux\leq 0\quad \text{and}\quad \rho \ux + y_j > 0.
\end{aligned}
\eeq
The utility function $u$ is of CRRA type:
$
u(c) = \frac{c^{1-\gamma}}{1-\gamma}.
$
The HJB equations can be rewritten as:
\begin{equation}\label{eq:HJB}
\rho v_j(x) = H(x, y_j, Dv_j) + \lambda_j (v_{\bar{\jmath}}(x) - v_j(x)),
\end{equation}
\begin{equation}\label{def H}
	H(x, y_j, p) = \sup_{c \ge 0} \left\{ u(c) + (r x + y_j - c)p \right\} =
\begin{cases}
	(rx + y_j)p + \dfrac{\gamma}{1 - \gamma} p^{1 - \frac{1}{\gamma}}, & \text{if } p \ge 0, \\
	+\infty, & \text{if } p < 0.
\end{cases}
\end{equation}
The Hamiltonian takes finite values only if $p\geq 0$. From an economic point of the view, the model only makes sense if the marginal value of wealth remains non-negative. The unboundedness of $H$ does not pose a problem for viscosity solution theory, as explained in \cite[p. 87]{barles1994solutions}.
\begin{remark}
We now give the rationale for the standing assumptions \eqref{standing assump}. The concavity of $u(c)$ requires $\gamma>0$. In this paper we focus on the case when risk aversion $\gamma>1$ as the solution is bounded. The cases $0<\gamma<1$ or the log-utility can be dealt with similarly but with additional technicalities. Although the interest rate $r$ is an equilibrium object, i.e.  not known a priori—we can restrict to the case $\rho > r$. This  is justified since from \cite[Proposition 4, p. 66]{achdou2022income},
$
\lim_{\substack{r\to \rho\\ r<\rho}}  K[m] =+\infty,\quad \lim_{r\to -\infty} K[m] =\ux.
$
Therefore the only possible equilibrium $r$ for the Huggett model is such that $r<\rho$ and bounded below. For the Aiyagari model, $(iii_A)$ \eqref{Aiyagari} implicitly assumed $K[m]> 0$ and $-\delta\leq r<\rho$. From $\ux\leq 0$, $r<\rho$ and $\rho \ux + y_j > 0$, it follows that $r\ux + y_j > 0$ and the admissible set of controls is non-empty at $\ux$.
\end{remark}
We observe that $H(x, y_j , p)$ is strictly convex in $p$ for fixed $(x, y_j)$ and coercive w.r.t. $p$ :
\begin{equation}\label{coerc}
	\lim_{p\to +\infty}H(x, y_j, p) =
\begin{cases}
	+\infty, & \text{if } rx+y_j> 0, \\
	-\infty, & \text{if } rx+y_j\leq  0.
\end{cases}	
\end{equation}
\begin{lemma}\label{lemma:H}
Let $rx+y_j>0$, then
\begin{equation}\label{H min}
	\min_{p > 0} H(x, y_j , p) = \frac{(rx + y_j )^{1 - \gamma}}{1 - \gamma}, \qquad \mathop{\arg\min}\limits_{p > 0} H(x, y_j , p) = (rx + y_j )^{-\gamma}.
\end{equation}
\end{lemma}

We recall the definition of constrained viscosity solution for the HJB system \eqref{eq:HJB}.
\begin{definition}\hfill\label[definition]{def vis sol}
	\begin{itemize}
		\item[1.] An upper semicontinuous (u.s.c.) function $v=(v_1,v_2)$ is said to be a viscosity subsolution of \eqref{eq:HJB}, if whenever $\varphi$ is smooth function, $j=1,2$ and $v_j-\varphi$ has a local maximum at $x_0$, then
	\[\rho v_j(x_0)\le H(x_0,y_j,D\phi(x_0) )+\lambda_j (v_{\bar \jmath}(x_0)-v_j(x_0)).\]	
		\item[2.] A  lower semicontinuous (l.s.c.) function $v=(v_1,v_2)$ is said to be a viscosity supersolution of \eqref{eq:HJB}, if whenever $\varphi$ is smooth function, $j=1,2$ and $v_j-\varphi$ has a local  minimum at $x_0$, then
	\[\rho v_j(x_0)\ge H(x_0,y_j,D\varphi(x_0) )+\lambda_j (v_{\bar \jmath}(x_0)-v_j(x_0)).\]		
	\end{itemize}
A continuous function 	$v$ is said to be a constrained viscosity solution to  \eqref{eq:HJB} if $v$ is a viscosity supersolution in $(\ux,\infty)$ and a viscosity subsolution in $[\ux,\infty)$.
\end{definition}
Now, we establish the strong comparison principle for system \eqref{eq:HJB}, stating that a u.s.c. sub solution is lower than a l.s.c. supersolution. The proof follows \cite[7.1.2 proof of Theorem 4.6]{barles1994solutions}. 
\begin{theorem}\label{comparison}
Assume that $\mathsf{u}=(\mathsf{u}_1,\mathsf{u}_2)$ and $\mathsf{v}=(\mathsf{v}_1,\mathsf{v}_2)$ are bounded viscosity sub and supersolution of system \eqref{eq:HJB}.We extend $\mathsf{v}_j$ at $\ux$ as 
\beq\label{liminf super}
\mathsf{v}_j(\ux)=\liminf_{\substack{z\rightarrow \ux\\ z>\ux}} \mathsf{v}_j(z) .
\eeq
Then $\mathsf{u}\leq \mathsf{v}$ in $[\ux,+\infty)$, i.e. $\mathsf{u}_j\le \mathsf{v}_j$ in $[\ux,+\infty)$ for $j=1,2$.
\end{theorem}
\begin{proof}

We assume by contradiction that 
\beq\label{Eq:contra}
\max_j\sup_x(\mathsf{u}_j(x)-\mathsf{v}_j(x))=\delta>0.
\eeq
First we consider the case when the $\sup$ in \cref{Eq:contra} is achieved at $\ux$, such that
\beq\label{Eq:contra-bd}
\max_j\sup_x(\mathsf{u}_j(x)-\mathsf{v}_j(x))=\max_j(\mathsf{u}_j(\ux)-\mathsf{v}_j(\ux))=\delta.
\eeq
From \cref{liminf super}, there  exists a sequence $\zeta_k\to \ux$, such that $\mathsf{v}_j(\zeta_k)\to \mathsf{v}_j(\ux)$. We denote $\e_k=\vert \zeta_k-\ux\vert$. Consider the   function
\beq
\psi_k(j,x,z)=\mathsf{u}_j(x)-\mathsf{v}_j(z)-\frac{\vert x-z\vert^2}{\e_k}-\left[\left(\frac{z-x}{\e_k}-1\right)_{-}\right]^2-\vert z-\ux\vert^2.
\eeq
Let $\psi_k$ attain its maximum at $(j_{k}, x_{k},z_{k})$. From $\psi_k(j_{k},x_{k},z_{k})\geq \psi_k(j_{\bar{\jmath}_{k}},x_{k},z_{k})$, 
$$
\mathsf{u}_{j_{k}}(x_{k})-\mathsf{v}_{j_{k}}(z_{k})\geq \mathsf{u}_{\bar{\jmath}_{k}}(x_{k})-\mathsf{v}_{\bar{\jmath}_{k}}(z_{k}),
$$
we can then obtain
\beq\label{v1 v2}
\mathsf{u}_{\bar{\jmath}_{k}}(x_{k})-\mathsf{u}_{j_{k}}(x_{k})\leq \mathsf{v}_{\bar{\jmath}_{k}}(z_{k})-\mathsf{v}_{j_{k}}(z_{k}).
\eeq
We now show 
\beq\label{psi_k}
\psi_k(j_{k},x_{k},z_{k})\geq \delta-o(1)>0. 
\eeq

From
$
\left[\left(\frac{\zeta_k-\ux}{\e_k}-1\right)_{-}\right]^2=0,
$
We have 
$$
\psi_k(j_{k}, x_{k},z_{k})\geq \max_j\psi_k(j,\ux,\zeta_k) =\max_j(\mathsf{u}_j(\ux)-\mathsf{v}_j(\zeta_k))-\vert \ux-\zeta_k\vert-\vert \zeta_k-\ux\vert^2.
$$
From $\zeta_k\to \ux$ and $\mathsf{v}_j(\zeta_k)\to \mathsf{v}_j(\ux)$ we obtain
$
\max_j\psi_k(j,\ux,\zeta_k) =\delta-o(1)
$
and therefore \cref{psi_k}. With $\psi_k(j_{k}, x_{k},z_{k})> 0$ and boundedness of $\mathsf{u}_{j_{k}}(x_{k})$, $\mathsf{v}_{j_{k}}(z_{k})$ there exists a constant $C>0$ such that 
$
\frac{\vert x_{k}-z_{k}\vert^2}{\e_k}<C,
$
hence $x_{k}-z_{k}\to 0$ as $\e_k\to 0$. From \cref{psi_k} we can obtain 
$$
\liminf_{k\to +\infty}(\mathsf{u}_{j_{k}}(x_{k})-\mathsf{v}_{j_{k}}(z_{k}))\geq \liminf_{k\to +\infty}\psi_k(j_{k},x_{k},z_{k})\geq \delta. 
$$
In the meantime, since $\mathsf{u}_{j_{k}}(x_{k})-\mathsf{v}_{j_{k}}(z_{k})$ is u.s.c., we have 
$$
\limsup_{k\to +\infty}(\mathsf{u}_{j_{k}}(x_{k})-\mathsf{v}_{j_{k}}(z_{k}))\leq \max_j(\mathsf{u}_j(\ux)-\mathsf{v}_j(\ux))=\delta,
$$
hence
$
\lim_{k\to +\infty}(\mathsf{u}_{j_{k}}(x_{k})-\mathsf{v}_{j_{k}}(z_{k}))=\delta. 
$
We then obtain 
\begin{equation*}
\frac{\vert x_{k}-z_{k}\vert^2}{\e_k}+
\left[\left(\frac{z_{k}-x_{k}}{\e_k}-1\right)_{-}\right]^2+
\vert z_{k}-\ux\vert^2\to 0,\qquad \text{as}\quad \e_k\to 0.
\end{equation*}
This gives 
$
z_{k}-x_{k}\geq \e_k-\e_ko(1),
$
which implies $z_{k}>\ux$, hence we can use the supersolution property at $z_{k}$. We denote
\begin{align*}
&\phi(x,z,\e)=\frac{\vert x_{\e }-z_{k}\vert^2}{\e }+
\left[\left(\frac{z_{\e }-x_{\e }}{\e }-1\right)_{-}\right]^2+
\vert z_{\e }-\ux\vert^2\\
&\Lambda_k=\frac{2}{\e_k}\left(\frac{z_{k}-\ux}{\e_k}-1\right)_{-}.
\end{align*}
Since $\psi_k(j_{k},x,z)$ attain it maximum at $(x_{k},z_{k})$, we have
\begin{align*}
\rho \mathsf{u}_{j_{k}}(x_{k}) \leq &H\left(x_{k},y_{j_{k}},\frac{2( x_{k}-z_{k} )}{\e_k}+\Lambda_k \right)+\lambda_{j_{k}}\left(\mathsf{u}_{\bar{\jmath}_{k}}-\mathsf{u}_{j_{k}}\right),\\
\rho \mathsf{v}_{j_{k}}(z_{k}) \geq & H\left(z_{k},y_{j_{k}},\frac{2( x_{k}-z_{k} )}{\e_k}+\Lambda_k-2(z_{k}-\ux) \right)+\lambda_{j_{k}}\left(\mathsf{v}_{\bar{\jmath}_{k}}-\mathsf{v}_{j_{k}}\right),
\end{align*}
with $\frac{2( x_{k}-z_{k} )}{\e_k}+\Lambda_k-2(z_{k}-\ux)\geq 0$. \par
Subtracting the two inequalities and using \cref{v1 v2}, we have 
\beq
\begin{aligned}
&\rho (\mathsf{u}_{j_{k}}(x_{k})-\mathsf{v}_{j_{k}}(z_{k}))\\
\leq {}&H\left(x_{k},y_{j_{k}},\frac{2( x_{k}-z_{k} )}{\e_k}+\Lambda_k \right)-H\left(z_{k},y_{j_{k}},\frac{2( x_{k}-z_{k} )}{\e_k}+\Lambda_k-2(z_{k}-\ux) \right)\\
\leq {}& r(x_{k}-z_{k})\left(\frac{2( x_{k}-z_{k} )}{\e_k}+\Lambda_k\right)+2(rz_{k}+y_{j_{k}})(z_{k}-\ux)\\
&+\underbrace{\frac \gamma {1-\gamma}\left(\frac{2( x_{k}-z_{k} )}{\e_k}+\Lambda_k \right)^{1-\frac 1 \gamma}
-\frac \gamma {1-\gamma}\left(\frac{2( x_{k}-z_{k} )}{\e_k}+\Lambda_k-2(z_{k}-\ux) \right)^{1-\frac 1 \gamma}}_{<0}\\
\leq {}& r(x_{k}-z_{k})\left(\frac{2( x_{k}-z_{k} )}{\e_k}+\Lambda_k\right)+2(rz_{k}+y_{j_{k}})(z_k-\ux).
\end{aligned}
\eeq
For the last inequality we used the fact that $\frac \gamma {1-\gamma}p^{1-\frac 1 \gamma}$ is a decreasing function when $p\geq 0$ and $z_k>\ux$. To obtain a contradiction by sending $\e_k \to 0$, the crucial point is to show $\vert x_{k}-z_{k}\vert \Lambda_k\to 0$. For this we refer to the coercivity of $H$. Suppose $rx+y_j>0$ and for the supersolution property to hold, there exists a $C>0$ such that 
$$
\frac{2( x_{k}-z_{k} )}{\e_k}+\Lambda_k-2(z_{k}-\ux) \leq C,
\quad
\Lambda_k \leq C+\frac{2( z_{k}-x_{k} )}{\e_k}+2(z_{k}-\ux).
$$
$\vert x_{k}-z_{k}\vert \Lambda_k\to 0$ clearly follows from $\frac{\vert x_{k}-z_{k}\vert^2}{\e_k}\to 0$ and $ z_{k}-\ux \to 0$. If $rx+y_j\leq 0$ We can argue similarly with the subsolution property and obtain $\vert x_{k}-z_{k}\vert \Lambda_k\to 0$.\par
We obtain $\mathsf{u}_{j_{k}}(x_{k})-\mathsf{v}_{j_{k}}(z_{k})\to 0$ as $\e_k\to 0$, which is a contradiction to \cref{psi_k}. Therefore we have shown that \cref{Eq:contra-bd} cannot hold. More generally to obtain a contradiction with \cref{Eq:contra} we can proceed similarly as \cite[proof of Theorem 2.4 and Theorem 4.2]{barles1994solutions} and the techniques developed above for the particular structure of the Hamiltonian.
\end{proof}
The result from \cref{comparison} apply directly to the decoupled system by taking $\lambda_j=0$:
\begin{equation}\label{eq:HJB1d}
\rho v_j(x) = H(x, y_j, Dv_j).
\end{equation}


\begin{proposition}\label{comparison1d}
The constant function 
\begin{equation}\label{sub sol}
	\widehat{\su}_j(x) = \frac{1}{\rho} \frac{(r\ux + y_j)^{1 - \gamma}}{1 - \gamma}
\end{equation}
is a subsolution  of \eqref{eq:HJB1d}. The constant $\widehat{\sv}_j=0$ is a supersolution to \cref{eq:HJB1d}. 
\end{proposition}
\begin{proof}
The supersolution $0$ follows from $H(x,y_j,p)=0$ if $p=0$. For $x>\ux$,

$
\rho \widehat{\sv}_j(x)=H(x,y_j,D\widehat{\sv}_j)=0,
$
hence $\widehat{\sv}_j$ is a supersolution. From
$$
\rho \widehat{\su}_j(x)<0=H(x,y_j,D\widehat{\su}_j(x)),\quad x>\ux,
$$
$\widehat{\su}_j$ is a subsolution for $x>\ux$. We now check that $\widehat{\su}_j$ is a subsolution at $x=\ux$. For any $\varphi \in C^1[\ux,+\infty)$, $\widehat{\su}_j-\varphi$ has a local maximum at $\ux$ iff $D\varphi(\ux)\geq D\widehat{\su}_j(\ux)$. From \cref{H min},  we still have $H(\ux,y_j,D\varphi(\ux))\geq \frac{(r\ux + y_j )^{1 - \gamma}}{1 - \gamma}$, hence $\rho \widehat{\su}_1(\ux)\leq H(\ux,y_1,D\varphi(\ux))$. 
\end{proof}

We now construct explicit upper and lower barrier functions for the problem \eqref{eq:HJB}.
\begin{proposition}\label{barrier}
The constant function $(\widehat{\su}_1, \widehat{\su}_1)$ is a subsolution  of \cref{eq:HJB}. The constant $(0,0)$ is a supersolution of \cref{eq:HJB}. 
\end{proposition}
We omit the proof since it is very similar to the proof of \cref{comparison1d} and in addition using $y_2>y_1$. 
\begin{corollary}
	There exists at most one bounded viscosity solution $v=(v_1,v_2)$ to the system \eqref{eq:HJB}.  
\end{corollary}
We do not address the existence of a solution to \eqref{eq:HJB} in detail here. However, combining the strong comparison principle (Theorem~\ref{comparison}) with the construction of barrier functions (Proposition~\ref{barrier}) allows one to establish existence via the Perron method. Alternatively, existence can also be obtained as the limit of the approximation scheme analyzed in the next section, see \cref{thm:convergence}.\par
We now consider some additional properties of the bounded viscosity solution to \eqref{eq:HJB}, which follows easily from the comparison principles and barrier properties. 


\begin{proposition}\label{solution property}
Let $v=(v_1,v_2)$ be the bounded viscosity solution to  \eqref{eq:HJB}. Then $v_j$ is Lipschitz continuous in $[\ux,+\infty)$. 
	
\end{proposition}
\begin{proof}
We consider $x,z\in [\ux,+\infty]$ and and aim to show that there exists a constant \( C \) such that  
	\begin{equation}\label{eq:claim_L1}
		\vert v_j(x) -  v_j(z) \vert\leq C |x - z|,\quad j=1,2.
	\end{equation}
	Consider the problem $\min_{z}v_j(z)+C\vert x-z\vert$, we show that the minimizer $\bar{z}$ coincides with $x$. Assume $\bar{z}\neq x$, then $\vert x-z\vert$ is differentiable and by definition of the viscosity supersolution 
$$
	\rho v_j(\bar{z}) \geq H \left( \bar{z}, y_j, -C \frac{\bar{z} - x}{|\bar{z} - x|} \right)+\lambda_j(v_\bjm(\bar{z}) -v_j(\bar{z}) ).
$$	
	If \( \bar{z} \neq \underline{x} \), we must have \( \bar{z} < x \), since otherwise  	
	$
	H \left( \bar{z}, y_j, -C \frac{\bar{z} - x}{|\bar{z} - x|} \right) = +\infty.
	$	
	This implies a contradiction with boundedness of $v$. When $\bar{z} < x $ and for sufficiently large \( C \), we also obtain a contradiction with  boundedness of $v$ from the coercivity of Hamiltonian \eqref{def H}. Therefore, for such \( C \), we conclude that \( \bar{z} =x \), leading to  	
	$$
	v_j(z) + C |z - x| \geq v(\bar{z}) + C |\bar{z} - x| = v_j(x),\quad v_j(x)- v_j(z)\leq C |z - x|.
	$$
	By symmetry we can also obtain $v_j(z)- v_j(x)\leq C |z - x|$. This establishes \eqref{eq:claim_L1}. 
	\end{proof}

We observe that if we assume $r\geq 0$ then we can obtain more intuitive sub and supersolutions. These results further justify using \cref{def vis sol} to select the ``right solution'' from the economic modeling point of view. 
\begin{proposition}\label{comparison1d r}
Let $r> 0$. The function
\begin{equation}\label{sub sol}
	\check{\su}_j(x) = \frac{1}{\rho} \frac{(rx + y_j)^{1 - \gamma}}{1 - \gamma}
\end{equation}
is a subsolution  of \eqref{eq:HJB1d}. The function 
\begin{equation}\label{super sol}
	\check{\sv}_j(x) =\left( \frac{\rho - r}{\gamma} + r \right)^{-\gamma} \frac{(x + y_j/r)^{1 - \gamma}}{1 - \gamma}
\end{equation}
is  a supersolution  of \eqref{eq:HJB1d}.
\end{proposition}
\begin{proof}
For $x>\ux$, it is easy to check $\check{\sv}_j$ satisfies \cref{eq:HJB1d} in the classical sense.\par
Since $D\check{\su}_j(x)=\frac{r}{\rho} (rx + y_j)^{-\gamma}$ and $\rho>r$, we can obtain from \cref{H min} that 
$$
H(x,y_j,D\check{\su}_j)>\frac{(rx + y_j )^{1 - \gamma}}{1 - \gamma}=\rho \check{\su}_j(x),\quad \forall x>\ux.
$$
The fact that $\check{\su}_j$ is a subsolution at $x=\ux$ has been shown in the proof of \cref{comparison1d}. 
\end{proof}

\begin{remark}
The function $\check{\sv}_j$ corresponds to the value function for an agent under the natural borrowing constraint $\underline{x} = -y_j/r$, it is the ``complete market solution'' to \cref{eq:HJB1d}.
It is interesting to understand why $\check{\sv}_j(x)$  is not a subsolution at $\ux$ in the sense of \cref{def vis sol}.  From $\rho > r$, $\gamma > 1$, we have  $\frac{\rho - r}{\gamma} + r > 0$ and can then obtain the inequality
\begin{equation}\label{ineq}
	\left( \frac{\rho - r}{\gamma} + r \right)^{-\gamma} r^{\gamma - 1} \leq \frac{1}{\rho}.
\end{equation}
To prove \eqref{ineq}, define the function
$
f(r) = \rho \left( \frac{\rho - r}{\gamma} + r \right)^{-\gamma} r^{\gamma - 1}.
$
Clearly $f(\rho)=1$. By simple computation we show $f'(r)>0$ with $\rho>r$ and $\gamma>1$. Finally $f(r)<1$ for $r<\rho$ and we obtain \eqref{ineq}. 
We consider the test function
$$
\varphi(x) = \frac{r^\gamma (x + y_j/r)^{1 - \gamma}}{1 - \gamma}.
$$
Since $\rho > r$ and $\gamma > 1$, we have
$
r^{-\gamma} > \left( \frac{\rho - r}{\gamma} + r \right)^{-\gamma}.
$
Therefore,
$$
D\varphi(\underline{x}) = r^{-\gamma} (\underline{x} + y_j/r)^{-\gamma} > \left( \frac{\rho - r}{\gamma} + r \right)^{-\gamma} (\underline{x} + y_j/r)^{-\gamma} = \check{\sv}_j(\underline{x}).
$$
This implies that $\underline{x}$ is a local maximum of $\check{\sv}_j- \varphi$. Applying \eqref{ineq} and \eqref{H min}, we find
$$
\rho \check{\sv}_j(\underline{x}) > \frac{(r\underline{x} + y_j)^{1 - \gamma}}{1 - \gamma} = \max_{c \ge 0} \left\{ u(c) + (r\underline{x} + y_j - c) D\varphi(\ux) \right\}.
$$
\end{remark}
\begin{remark}
With $r>0$, we can also use \cref{ineq} to check that 
$\check{\su}_j(x)\leq \check{\sv}_j(x)$ for all $x\in [\ux,+\infty)$. This is an example of the comparison principle. 
\end{remark}
For the rest of the section we assume the viscosity solution $v_j$ is concave. It is important to notice this is justified if we assume $v_j$ is the value function of the optimal control problem. The equivalence between value function and the unique state constraint viscosity solution is standard (c.f. \cite{Soner_1986}), but not proven in this paper. We plan to do it in our future works. 
\begin{proposition}
Let $v=(v_1,v_2)$ be the solution to the system \eqref{eq:HJB} and assume $v_j$ is concave. We have $v_j$ is $C^1$ in $(\ux,+\infty)$. In particular, $Dv_j$ is uniformly continuous in $[\ux,R]$ for any constant $R>\ux$. 
\end{proposition}
\begin{proof}
From the Lipschitz continuity, $v_j$ is differentiable almost everywhere in $(\ux,+\infty)$. By using the strict convexity of $H(x,y_j,p)$ (see \eqref{def H}) in the $p$ variable, with the same arguments from \cite[Section 5.2, Proposition 5.7]{bardi1997optimal}, we can in fact show that $v_j$ is $C^1$ in $(\ux,+\infty)$. From coercivity, we obtain that $Dv_j(x)$ is uniformly bounded for $x>\ux$. \par
From the concavity of $v_j$, $Dv_j(x)$ is monotone increasing as $x\to \ux$. There exists $Dv_j(\ux^+)$ such that $Dv_j(\ux^+)=\lim_{x\to \ux,x>\ux}Dv_j(x)$. Moreover, we define $Dv_j(\ux)=\lim_{\epsilon \to 0,\epsilon>0}\frac{v_j(\ux+\epsilon)-v_j(\ux)}{\epsilon}$. Since $v_j$ is $C^1$ in $(\ux,R]$ and continuous in $[\ux,R]$, we obtain $Dv_j(\ux)=Dv_j(\ux^+)$ by using finite increment theorem. Therefore, $Dv_j$ is continuous on the interval $[\ux,R]$ and we obtain the uniform continuity by Heine-Borel theorem. 
\end{proof}


We consider some stability properties w.r.t. the interest rate $r$. 
\begin{proposition}\label{stability}
Assume that the solution $v^\i$ to system \eqref{eq:HJB} corresponding to an interest rate $r^\i$ is concave. For $r^\i \to r$, the sequence $v^\i$ converges in $C^1[\ux,R]$ to $v$ for any constant $R>\ux$. 
\end{proposition}
\begin{proof}
By stability property of constrained viscosity solution we have $v^\i$ converges to $v$ uniformly. Since $v_j^\i$ is concave, we have  $D v^\i_j(x)$ converges to $Dv_j(x)$ pointwise for $x\in [\ux,+\infty)$ (\cite[Theorem 3.3.3]{cannarsa2004semiconcave}). The local uniform convergence can be proved as in \cite[Lemma 5.3]{achdou2023mean}, using the concavity of $v^\i_j$ and uniform continuity of $Dv_j$. 
\end{proof}
We denote by $c_j^{\i,*}$ and $s_j^{\i,*}$ the optimal consumption and saving policies when $r=r^\i$. From Proposition \ref{stability}, the sequences $c_j^{\i,*}$ and $s_j^{\i,*}$ converge locally uniformly to $c^*_j$ and $s^*_j$ as $r^\i\to r$ .

The following theoretical results are again based on \cite[Proposition 1 and 2]{achdou2022income}. The proof is based on considering $v=(v_1,v_2)$ as the value function of the optimal control problem \eqref{control_problem}. We give it here in order to justify using state constraint boundary condition with a sufficiently large $x_{\max}$ such that $x_{\max}>\bar{x}$, while designing the numerical algorithms on the domain $[\ux,x_{\max}]$.   
\begin{proposition}
The saving policy for the low income type satisfies $s^*_1(x)\leq 0$ for all $x\in [\ux,+\infty)$. In particular $s^*_1(\ux)=0$. There exists $\ux\leq \bar{x}<+\infty$ such that 
$s_2^*(x)<0$ for all $x>\bar{x}$ and 
$s_2^*(x)>0$ for all $\ux<x<\bar{x}$.
Moreover $\mu_2=0$ if $\bar{x}>\ux$. 
\end{proposition}

\section{The  semi-Lagrangian scheme}
In this section we introduce the approximation schemes for systems \eqref{MFG} and \eqref{Dynamic MFG}. For the stationary system \eqref{MFG} we fix a step $h$ and we consider a discrete in time model which evolves at time $nh$, $n\in\N$.
\subsection{The discrete Hamilton-Jacobi-Bellman equation}
The dynamics of the representative agent is given by
\begin{equation}\label{eq:discrete_dynamics}
\left\{\begin{aligned}
\qquad & x_{n+1}=x_n+h(rx_n+y_n-c_n)\delta_{y_n,y_{n+1}}\\
&		x_0=x\ge \ux,\quad x_n\geq x
\end{aligned}\right.
\end{equation}
Here $y_n$ is Poisson process such that $\PP(y_{n+1}=y_\bjm|y_n=y_j)=\lambda_j h$ and
$\delta_{y,\bar y}=1$ if $y=\bar y$ and  $\delta_{y,\bar y}=0$ otherwise. Each agent  maximizes the cost functional 
\begin{equation*}
J^{h}_{j} (\{c_n\}; x)=\EE_{x,j}\left[	\sum_{n=0}^\infty h(1-\rho h)^n u(c_n)\right].
\end{equation*}
The corresponding value function is
$
	v^{h}_{j}(x)=\sup_{\{c_n\}\in\mC_j^h(x)}J^{h}_{j} (\{c_n\}; x)
$
where (c.f. \cite[Eq. (3.2)]{achdou2022simulating})
\begin{equation}\label{C set}
\mC^{h}_{j}(x):=\left\{{c}: c\ge 0\quad\text{and}\quad x+h(rx+y_j-c) \ge \ux \right\}.
\end{equation}
By   the Dynamic Programming Principle we get the HJB equation for $j=1,2$
\begin{equation}\label{eq:HJBh_1}
\begin{aligned}
	v^{h}_{j}(x)
	=\sup_{c\in \mC^{h}_{j}(x)}\left\{
	(1-\rho h)(1-\lambda_jh)v^{h}_{j}(x+h(rx+y_j-c))+ hu(c) 
	\right\}+(1-\rho h)\lambda_j hv^{h}_{\bar \jmath}(x),
	\end{aligned}
\end{equation}
or equivalently
\begin{equation}\label{eq:HJBh_2}
\begin{aligned}
\rho v^{h}_{j}(x)
={}&\sup_{c\in \mC^{h}_{j}(x)}\left\{ 
		(1-\rho h)  (1-\lambda_jh)\frac{v^{h}_{j}(x+h(rx+y_j-c))-v^{h}_{j}(x)}{h}+u(c)\right\}\\
		&+(1-\rho h)\lambda_j (v^{h}_{\bar \jmath}(x)-v^{h}_{j}(x)) , 
\end{aligned}
\end{equation}
where $\mC^{h}_{j}(x)$ is defined as in \eqref{C set}. We denote by $c^*_j(x) \in \mC^{h}_{j}(x)$ a control that attains the maximum in \eqref{eq:HJBh_2}.\par
The fully discrete scheme for the HJB equation is obtained by projecting the equation \eqref{eq:HJBh_1}, or \eqref{eq:HJBh_2}, on a grid. Fix $\Dx>0$ and set $\Delta=(h,\Dx)$. Let $x_i=\ux+i \Dx$, $i\in\NN$, be the points of the space grid. Consider a   $\mathbb{Q}_1$ basis  $(\beta_{i})_{i\in \NN}$, where   $\beta_i$  is a polynomial of degree less than or equal to $1$ and satisfies that  $\beta_i(x_k)=1$ if $i=k$ and $\beta_i(x_k)=0$, otherwise. Moreover, the support $\mbox{supp}(\beta_i)$ of $\beta_i$ is compact and  
\begin{equation}\label{eq:baricentric}
	0\leq \beta_i \leq 1 \; \; \;  \; \forall \; i\in \NN, \qquad\sum_{i\in  \NN} \beta_i(x)=1 \hspace{0.3cm} \forall \; x\in [\ux,\infty).
\end{equation}
Denote with $\mA^\Dx =\{x_i\}_{i\in\NN}$ the set of the vertices of the grid and
  $B(\mA^\Dx)$ be the space of bounded functions on $\mA^\Dx$. For $\phi\in B(\mA^\Dx)$, set $ \phi_{i}$ be its value at $x_i$. We consider the following  linear interpolation operator 
\begin{equation}\label{Eq:intrep}
	I[\phi](\cdot):=\sum_{i\in  \NN} \phi_i\beta_i(\cdot) \qquad \mbox{for } \phi \in B(\mA^\Dx) .
\end{equation}
We look for a  function $\bV \in  B(\mA^\Dx) $ which solves  \eqref{eq:HJBh_1}   at the vertices $x_i$  of the grid. We get the fully discrete Hamilton-Jacobi-Bellman equation
\begin{equation}\label{eq:HJBhD_1}
\begin{aligned}
&V^{\D}_{i,j}	\\
={}&\sup_{c\in \mC^{\D}_{j}(x_i)}\left\{
		hu(c)+(1-\rho h)\left[\lambda_j hV^{\D}_{i,\bjm} 
		+(1-\lambda_jh)\left(\sum_k \underbrace{\beta_k(x_i+hs_{i,j}(c))}_{M_{i,k}}V^{\D}_{k,j}\right)\right]
		\right\},\\
& s_{i,j}(c)=rx_i+y_j-c,\\
\end{aligned}
\end{equation}
where we define $V^{\D}_{k,j}= v^{\D}_{j}(x_k)$, $k\in \N$. We consider a matrix $M$ such that the $(i,k)-$ entry $M_{i,k}=\beta_k(x_i+hs_{i,j}(c))$, then $0\le M_{i,k}\le 1 $, $\sum_k M_{i,k}=1$ for all $i\in\N$. For any given $s\in \R$, there are only two non zero entries of the vector $M_{i}=(M_{i,k})_{k\in\N}$, i.e. $\beta_k(x_i+hs)$ and $\beta_{k+1}(x_{i}+hs)$ for $k$ such that $x_i+hs \in [x_k, x_{k+1}]$. Next, we denote by $\mathbf{c}_j$ a vector with elements $c_j(x_i)$. We denote the vector $\bV_j=V^{\D}_{\cdot,j}$. We can then write  \cref{eq:HJBhD_1} in vector form
\begin{equation}\label{eq:HJB_vect}
\bV_j=\sup_{\mathbf{c}\in \mC^{\D}_{j}}\left\{
hu(\mathbf{c})+(1-\rho h)\left[\lambda_j h\bV_{\bjm} 
+(1-\lambda_jh)\left(M(\mathbf{s}_j(\mathbf{c}))\bV_j\right)\right]
\right\}\qquad \text{for }j=1,2.
\end{equation}
In the evolutive case, consider the finite horizon optimal control problem 
\begin{equation*}
J^{h}_{j} (\{c_n\}; x)=\EE_{x,j}\left[\sum_{n=0}^{N-1} h(1-\rho h)^n u(c_n)+v_{j_N}(x_N)\right],
\end{equation*}
we obtain similarly a fully discrete HJB equation, by denoting the drift $s^n_{i,j}(c)=r_nx_i+y_j-c$,
\begin{equation}\label{evoHJBhD_1}
\left\{\begin{aligned}
\,&		V^{\D,n}_{i,j}	=\sup_{c\in \mC^{\D,n}_{j}(x_i)}\Biggl\{
		hu(c)
		+(1-\rho h)\left[\lambda_j hV^{\D,n+1}_{i,\bjm} 
		+(1-\lambda_jh)\left(\sum_k \beta_k\left(x_i+hs^n_{i,j}(c)\right)V^{\D,n+1}_{k,j}\right)\right]
		\Biggl\},\\
& V^{\D,N}_{i,j}=V^{\D,st}_{i,j}.
\end{aligned}\right.
\end{equation}  
We can also write \cref{evoHJBhD_1} in vector form, with $\bV^n_{j}=V^n_{\cdot,j}$: for $j=1,2$,
\begin{equation}\label{evoHJB vector}
\bV^n_{j}=\sup_{\mathbf{c}\in \mC^{\D,n}_{j}}\left\{
hu(\mathbf{c})+(1-\rho h)\left[\lambda_j h\bV^{n+1}_{\bjm} 
+(1-\lambda_jh)\left(M\left(\mathbf{s}^n_{j}(\mathbf{c})\right)\bV^{n+1}_{j}\right)\right]
\right\}.
\end{equation}       
\begin{remark}
It is clear that the vectors $\bV_{j}$ and $\bV^n_{j}$ depend also on $\D$, just like $V^{\D}_{i,j}$ and $V^{\D,n}_{i,j}$. We drop this dependence to alleviate the notation. In particular, $\bV_{j}$ and $\bV^n_{j}$ are introduced mainly to illustrate the well posednes and implementation of our numerical algorithms for fixed $\D$. 
\end{remark}  
                                                                                     
\subsection{The approximate Fokker-Planck equation}
To approximate the FPK equation, we first consider   the semi-discrete problem and then we project it on a grid. We consider a measure $dm^h$ on $[\ux,\infty)\times \{y_1,y_2\}$ of the form $dm^h=\sum_{j\in\{1,2\}} dm^{h}_{j}(x)\otimes\delta_{y_j}(y)$ where $dm^{h}_{j}$ is a measure on $[\ux,\infty)$. Since $dm^h$ is an invariant measure for the discrete process
\eqref{eq:discrete_dynamics}  with the optimal control $\{c^*_{j}(x_n)\}_n$, we have that for any function $\Phi$ on $[\ux,\infty)\times \{y_1,y_2\}$ the identity
$$
\int_{x\geq \ux}\Phi(x_n,y_n) dm^h(x,y)=\mathbb{E}\left\{\int_{x\geq \ux} \Phi (x_{n+1},y_{n+1})dm^h(x,y)\right\},\quad \int_{x\geq \ux}dm^h(x,y)=1,
$$
where $(x_{n+1},y_{n+1})$ are given  as in \eqref{eq:discrete_dynamics}. We denote $s^*_{j}(x)=rx+y_j-c^*_{j}(x)$. Writing the previous relation component-wise for $\phi:[\ux,\infty)\to\R$, we get 
\beq\label{Eq:semi-FPK}
\begin{aligned}
	&\int_{x> \ux}\phi(x)g^{h}_{j}(x)dx+\mu_{j}\phi(x)\\
	={}&\int_{x\geq \ux}\phi(x+hs^*_{j}(x))\PP(y_{n+1}=y_j|y_n =y_j)dm_j
	+\int_{x\geq \ux}\phi(x)\PP(y_{n+1}=y_{\bar \jmath}|y_n =y_j)dm_j
	\\
	 ={}&(1-\lambda_jh)\left(\int_{x> \ux}\phi(x+hs^*_{j}(x))g^{h}_{j}(x)dx+\mu_{j}\phi(\ux+hs^*_{j}(\ux))\right)
	+\lambda_{\bar \jmath}h\left(\int_{x> \ux}\phi(x)g^{h}_{\bar \jmath}(x)dx+\mu_{\bar \jmath}\phi(\ux)\right).
\end{aligned}
\eeq
Now to get the fully discrete FPK equation, consider a measure $dm^\D=\sum_{j\in\{1,2\}}dm_j^{\D}(x)\otimes\delta_{y_j}(y)$ on $\mA^\Dx\times \{y_1,y_2\}$ where
$dm^{\D}_{j}(x)=\sum_{k\in\NN}\left(G^{\D}_{k,j} \delta_{x_k}(x)\right)\D x$ and, for $\phi\in B(\mA^\Dx)$, test the identity above with a  $I[\phi](x)=\sum_i\beta_i(x)\phi_i$. We get
$$
\sum_{i,k}\phi_i \beta_i(x_k)G^{\D}_{k,j}=(1-\lambda_jh)\sum_{i,k}\phi_i \beta_i\left(x_k+hs_{k,j}^*\right)G^{\D}_{k,j}+\lambda_{\bar \jmath} h\sum_{i,k}\phi_i \beta_i(x_k)G^{\D}_{k,\bar \jmath}
$$
for the arbitrariness of $\phi$ and recalling that $\beta_i(x_k)=\delta_i(k)$, we get the fully discrete FPK equation
\begin{equation}\label{eq:FP}
\left\{\begin{aligned}
&(i)\qquad &&	G^{\D}_{i,j}=(1-\lambda_jh)  \sum_k \beta_i\left(x_k+hs_{k,j}^*\right) G^{\D}_{k,j}    +\lambda_{\bar \jmath}hG^{\D}_{i,\bjm},\\
&(ii) &&	\sum_{i\in\N}G^{\D}_{i,1}+\sum_{i\in\N}G^{\D}_{i,2}=1/\D x.
	\end{aligned}\right.
\end{equation}
Recalling \cref{eq:mg}, $G^\D_{i,j}$ approximates the density $g_j$ at $x=x_i$ and $G^\D_{0,j}\D x$ approximates the weights $\mu_j$ on the Dirac mass at $\ux$ as $\D \to 0$. We have
$$
G^\D_{0,j}=\frac{2}{\D x}\int_{\ux}^{\ux+\frac{\D x}{2}}g_j(x)dx+\frac{\mu_j}{\D x},\quad G^\D_{i,j}=\frac{1}{\D x}\int_{x_i-\frac{\D x}{2}}^{x_i+\frac{\D x}{2}}g_j(x)dx\quad \text{if}\quad i\geq 1.
$$
We can write \cref{eq:FP} in vector form
\begin{equation}\label{eq:FP_vect}
	\bG_{j}=(1-\lambda_jh) \left(M(\mathbf{s}_j^*)\right)^{\texttt{T}}\bG_{j}    +\lambda_{\bar \jmath}h\bG_{\bjm},\quad \bG^{\texttt{T}}_{j}I+\bG^{\texttt{T}}_{\bjm}I=1/\D x.
\end{equation}
Next we show this scheme preserves the structural property \eqref{Ng} on the discrete level. Summing \eqref{eq:FPhD} ($i$) on $i$ and recalling that $\sum_i\beta_i(x)=1$, we get
$
\lambda_1\sum_i G^{\D}_{i,1}=\lambda_2\sum_i G^{\D}_{i,2}
$
and with \cref{eq:FP} ($ii$) we obtain
\begin{equation}\label{eq:cond_mass}
	\sum_i G^{\D}_{i,j}\D x=\frac{\lambda_{\bar \jmath}}{\lambda_1+\lambda_2} \qquad j=1,2.
\end{equation}

\begin{remark}
We do not prove the convergence of the scheme for FPK equation, but it is clear with Taylor expansion that as $h\to 0$ \cref{Eq:semi-FPK} gives the weak formulation of the FPK equation (c.f. \cite[Eq. (4.69), p. 296]{achdou2021mean}): for all test functions $(\phi_1,\phi_2)\in \left(C_c^1([\ux,+\infty))\right)^2$,
\begin{equation*}
\int_{x>\ux}\left(\lambda_j g_j(x)-\lambda_{\bar \jmath} g_{\bar \jmath}(x)\right)\phi_j(x)dx+(\lambda_j\mu_j-\lambda_{\bar \jmath}\mu_{\bar \jmath})\phi_j(\ux)
=\int_{x>\ux}s_j^*(x)D\phi_j(x)dx+\mu_js^*_j(\ux)D\phi_j(\ux).
\end{equation*}
\end{remark}

Similarly to the stationary case, we can derive the fully discrete (forward in time) FPK equation for approximating \eqref{Dynamic MFG} $(ii)$:
\begin{equation}\label{eq:FPhD}
\left\{\begin{aligned}
\qquad &G^{\D,n+1}_{i,j}=(1-\lambda_jh)\sum_k G^{\D,n}_{k,j}\beta_i(x_k+hs_{k,j}^{*,n})+\lambda_{\bar \jmath}hG^{\D,n}_{i,\bjm},\\
&G^{\D,0}_{i,j}=\mathsf{G}^{\D}_{i,j},
\end{aligned}\right.
\end{equation}
where
$$
\mathsf{G}^\D_{0,j}=\frac{2}{\D x}\int_{\ux}^{\ux+\frac{\D x}{2}}\mathsf{g}_j(x)dx+\frac{\mu_j(0)}{\D x},\quad \mathsf{G}^\D_{i,j}=\frac{1}{\D x}\int_{x_i-\frac{\D x}{2}}^{x_i+\frac{\D x}{2}}\mathsf{g}_j(x)dx\quad \text{if}\quad i\geq 1.
$$
\subsection{The approximate equilibrium system}
We obtain the fully discrete scheme for the stationary Mean Field Game system
\begin{equation}\label{eq:fully_discrete_scheme}
\left\{\begin{aligned}
&(i)&&		V^{\D}_{i,j}=\sup_{c\in \mC^{\D}_{j}(x_i)}\Biggl\{
		hu(c)
		+(1-\rho h)\left[\lambda_j hV^{\D}_{i,\bjm} 
		+(1-\lambda_jh)\left(\sum_k \beta_k(x_i+hs_{i,j}(c))V^{\D}_{k,j}\right)\right]
		\Biggl\},\\[6pt]
&&& s_{i,j}(c)=rx_i+y_j-c,\quad s^*_{i,j}=rx_i+y_j-c^*_{i,j},\\
&&& c^*_{i,j}=\mathop{\arg\max}\limits_{c\in \mC^{\D}_{j}(x_i)}\left\{
		hu(c)+(1-\rho h)(1-\lambda_jh)\left(\sum_k \beta_k(x_i+hs_{i,j}(c))V^{\D}_{k,j}\right)\right\},\\[6pt]	
&(ii) &&		G^{\D}_{i,j}=(1-\lambda_jh) \left( \sum_k \beta_i\left(x_k+hs^*_{k,j}\right) G^{\D}_{k,j}\right) +\lambda_{\bar \jmath}hG^{\D}_{i,\bjm},\\
&&&\sum_k G^{\D}_{k,1}\D x+\sum_k G^{\D}_{k,2}\D x=1,
\end{aligned}\right.
\end{equation}
for $j=1,2$, $i\in\N$. To pin down equilibrium $r$ we set:
$$
K[\bG]=\sum_k x_k G^{\D}_{k,1}\D x+\sum_k x_k G^{\D}_{k,2}\D x,\quad N[\bG]=\frac{y_1\lambda_2}{\lambda_1+\lambda_2}+\frac{y_2\lambda_1}{\lambda_1+\lambda_2},
$$
where we use \eqref{Aiyagari} ($iii_A$) for the Aiyagari model and $K[\bG]=B$ for the Huggett model.
\section{The convergence analysis}
In this section, we study the convergence properties of the scheme for the HJB equation, i.e. ($iii$) in system \eqref{eq:fully_discrete_scheme} with a fixed interest $r$ such that $r<\rho$. The results in this section apply to systems with different coupling conditions: Huggett, Aiyagari etc. \par

For a fixed $\Delta=(h,\Dx)$, we rewrite the scheme \eqref{eq:HJB_vect} as
\begin{equation}\label{Scheme}
	\mathcal{F}^\D_j\left(x_i,\left[V^{\Delta}_{i,j},V^{\Delta}_{i,\bar \jmath}\right],V^{\Delta}_{\cdot,j}\right)=0\qquad i\in\NN,\,j=1,2,\,\bar{\jmath}=3-j.
\end{equation}
where $\mathcal{F}^\D_j:\mA^\Dx\times\R^2\times B(\mA^\Dx)\to \R$ is defined by 
\begin{equation}\label{def S}
\begin{aligned}
\mathcal{F}^\D_j(x_i,(\bq_j,\bq_{\bar \jmath}),\sU)={}&\rho \bq_j-(1-\rho h)\lambda_j(\bq_{\bar \jmath}-\bq_j)\\
-\sup_{c\in \mC^{\D}_{j}(x_i)}&\left\{
 u(c)+(1-\rho h)(1-\lambda_jh)\frac{1}{h}\left(\sum_k \beta_k(x_i+hs_{i,1}(c))\sU_k-\bq_j  \right)\right\},
\end{aligned}
\end{equation}
where the set of controls $ \mC^{\D}_{j}(x_i)$ is defined as in \cref{C set}.
\begin{remark}\label{C_j}
We will use the condition for step size $\Dx\sim h$ for $\D\to 0$. Notice in the fully discretized setting \cref{C set} becomes  
\beq\label{Eq:C_j}
\mC^{\D}_{j}(x_i)=\left[0, \frac{x_i-\ux}{h}+rx_i+y_j\right].
\eeq
By imposing $\Dx\sim h$, the constraint is binding only at $\ux$. For any $x>\ux$, $x=i(\D x)\Dx$ where $i(\D x)\to +\infty$ as $\D x\to 0$. Therefore, for any given $R>0$, if $\D$ is sufficiently small and $x_i>\ux$ then $[0,R)\subset \mC^{\D}_{j}(x_i)$. 
\end{remark}
Similarly we define the linearized scheme, with fixed $c\in \mathbb{R}^+$
\begin{equation}\label{def linear S}
\begin{aligned}
F^\D_j(x_i,c,(\bq_j,\bq_{\bar \jmath}),\sU)&=\rho \bq_j-(1-\rho h)\lambda_j(\bq_{\bar \jmath}-\bq_j)\\
-&\left(
 u(c)+(1-\rho h)(1-\lambda_jh)\frac{1}{h}\left(\sum_k \beta_k(x_i+hs_{i,1}(c))\sU_k-\bq_j  \right)\
\right)
\end{aligned}
\end{equation}
The scheme \eqref{def linear S} is used for solving linearized HJB equation, i.e. holding $c_{i,j}$ fixed in system \eqref{eq:fully_discrete_scheme} ($i$). This will be particularly useful when we discuss the Howard algorithm for solving the HJB equations. \par
Next, we consider the monotonicity of the scheme. 
\begin{lemma}\label{monotone_scheme}
For any  $\D$, $i\in\N$,
	bounded functions $\sU$, $\sV\in B(\mA^\Dx)$ such that $\sU_k\le \sV_k$  $\forall k\in\N$ and  $(\bq_1,\bq_2)$, $(\bm_1,\bm_2)\in\R^2$ such that $\displaystyle \theta:=\bq_j-\bm_j=\max_{k=1,2}\{\bq_k-\bm_k\}\ge 0$, then
\begin{equation}\label{S monotone}
\mathcal{F}^\D_j(x_i,(\bq_j,\bq_{\bar \jmath}), \mathsf{U}+\theta)-\mathcal{F}^\D_j(i,(\bm_j,\bm_{\bar \jmath}),\mathsf{V})\ge \rho\theta.
\end{equation}
\end{lemma}
\begin{proof}
We recall that for all $i$, $j=1,2$ and $c \in \mC^{\D}_{j}$,
\beq\label{M pos}
\beta_k(x_i+hs_{i,1}(c))\ge 0, \quad \sum_k \beta_k(x_i+hs_{i,1}(c))=1.
\eeq
Assume that $\theta:=\bq_1-\bm_1\ge \bq_2-\bm_2$, hence $\bq_2-\bq_1\le \bm_2-\bm_1$. We first obtain from $\sU_k\le \sV_k$ for all $k$ and \cref{M pos} that
\beq\label{U leq V}
\begin{aligned}
&\sup_{c\in \mC^{\D}_{j}(x_i)}\left\{
	u(c)+(1-\rho h)(1-\lambda_1h)\frac{1}{h}\left(\sum_k \beta_k(x_i+hs_{i,1}(c))\sU_k-\bm_1\right)\right\}\\
\leq {}&
\sup_{c\in \mC^{\D}_{j}(x_i)}\left\{
	u(c)+(1-\rho h)(1-\lambda_1h)\frac{1}{h}\left(\sum_k \beta_k(x_i+hs_{i,1}(c))\sV_k-\bm_1\right)\right\}.
\end{aligned}
\eeq
By \cref{def S}, $\bq_1=\bm_1+\theta$ and \cref{M pos} we have
\begin{align*}
&\quad \mathcal{F}^\D_1(x_i,(\bq_1,\bq_2), \sU+\theta)\\
&=\rho(\bm_1+\theta)-(1-\rho h)\lambda_1(\bq_2-\bq_1)\\
       &-\sup_{\bc\in \mC^{\D}_{j}(x_i)}\left\{
	u(c)+\frac{(1-\rho h)(1-\lambda_1h)}{h}\left(\sum_k \beta_k(x_i+hs_{i,1}(c))(\sU_k+\theta)-(\bm_1+\theta)  \right)\right\}\\
	&\geq \rho \theta+\rho \bm_1-(1-\rho h)\lambda_1(\bm_2-\bm_1)\\
       &-\sup_{c\in \mC^{\D}_{j}(x_i)}\left\{
	u(c)+\frac{(1-\rho h)(1-\lambda_1h)}{h}\left(\sum_k \beta_k(x_i+hs_{i,1}(c))\sU_k-\bm_1\right)\right\}.
\end{align*}
We can then apply \cref{U leq V} to obtain \cref{S monotone}. We proceed similarly if $\theta:=\bq_2-\bm_2\ge \bq_1-\bm_1$. Hence monotonicity of the scheme is proved.\par
 \end{proof}
 
We now derive the discrete comparison principle for the scheme with the monotonicity property. 
\begin{definition}\label{def:sol_scheme}
	We say that $\mathbf{U}=(\bU_1,\bU_2)=(U_{\cdot,1},U_{\cdot,2})\in B(\mA^\Dx)^2$ is a subsolution (respectively, a supersolution) of \eqref{Scheme} if 
	\[	\mathcal{F}^\D_j(x_i,[U_{i,j},U_{i,\bar \jmath}],U_{\cdot,j})\le 0 \quad\text{(resp., $\ge 0$)}\qquad \forall i\in\N, \,j=1,2.\]
\end{definition}
\begin{proposition}\label{prop:comparison}
	If $\bU, \bV\in B(\mA^\Dx)^2$ are a bounded subsolution and, respectively, a bounded supersolution of \eqref{Scheme}, then $\bU\le \bV$, i.e. $U_{i,j}\le V_{i,j}$ $\forall i\in\N$, $j=1,2$.
\end{proposition}
\begin{proof}
	Assume by contradiction that $\theta=\sup_{i\in\N}\max_{j=1,2}\{U_{i,j}-V_{i,j}\}>0$. Consider a sequence $\theta_n\to \theta$ and $i_n\in\N$, $j_n\in \{1,2\}$ such that
	$
	\theta_n=U_{i_n,j_n}-V_{i_n,j_n}=\max_{j=1,2}\{U_{i_n,j}-V_{i_n,j}\}
	$.
	Exploiting  Def. \ref{def:sol_scheme},
$$
	 \mathcal{F}^\D_{j_n}(x_{i_n},(U_{i_n,j_n},U_{i_n,\bjm_n}),U_{\cdot,j_n})-\mathcal{F}^\D_{j_n}(x_{i_n},(V_{i_n,j_n},V_{i_n,\bjm_n}),V_{\cdot,j_n})\leq 0.
$$		
We then apply the monotonicity property \eqref{S monotone}, more specifically replacing $\sU$ by $U_{\cdot,j_n}-\theta$ and $\sV$ by $V_{\cdot,j_n}$ in \eqref{S monotone}, recalling $U_{\cdot,j_n}-\theta\leq V_{\cdot,j_n}$ :
\begin{align*}		
& \mathcal{F}^\D_{j_n}(x_{i_n},(U_{i_n,j_n},U_{i_n,\bjm_n}),U_{\cdot,j_n})-\mathcal{F}^\D_{j_n}(x_{i_n},(V_{i_n,j_n},V_{i_n,\bjm_n}),V_{\cdot,j_n})\\
={}&\mathcal{F}^\D_{j_n}(x_{i_n},(V_{i_n,j_n}+\theta_n,U_{i_n,\bjm_n}),U_{\cdot,j_n})-\mathcal{F}^\D_{j_n}(x_{i_n},(V_{i_n,j_n},V_{i_n,\bjm_n}),V_{\cdot,j_n})\\
\geq {}&\mathcal{F}^\D_{j_n}(x_{i_n},(V_{i_n,j_n}+\theta,U_{i_n,\bjm_n}+\theta-\theta_n),U_{\cdot,j_n}-\theta+\theta)-\mathcal{F}^\D_{j_n}(x_{i_n},(V_{i_n,j_n},V_{i_n,\bjm_n}),V_{\cdot,j_n})\\
&+(1-\rho h)(1-\lambda_1h)\frac{\rho}{h}(\theta_n-\theta)\\
\ge {}&\rho \theta+(1-\rho h)(1-\lambda_1h)\frac{\rho}{h}(\theta_n-\theta).		
\end{align*}
We get a contradiction by sending $\theta_n\to \theta$ for any fixed $h$. Therefore we have shown $\bU\le \bV$.	
\end{proof}

To solve the linear system with $c(x)\geq 0$:
\begin{equation}
		\rho v_j(x) = u(c) +(rx +y_j- c)Dv_j +\lambda_j(v_{\bar \jmath}(x)-v_j(x)),\quad j\in \{1,2\}\quad\text{and}\quad\bar \jmath=3-j,
	\end{equation}
we introduce 
\beq\label{linear scheme}
F^\D_j(x_i,\mathsf{c}_{i,j},[U_{i,j},U_{i,\bar{\jmath}}],U_{\cdot,j})=0 \qquad \forall i\in\N, \,j=1,2,\,\bar{\jmath}=3-j.
\eeq
This can be obtained from \cref{eq:fully_discrete_scheme} by holding the consumption policy $\mathsf{c}_{\cdot,j}\in \mC_j^h$ fixed rather than taking a $\sup$. The comparison principle holds also for $F^\D_j$. 
\begin{proposition}\label{linear comparison}
If $\bU, \bV\in B(\mA^\Dx)^2$ are a subsolution and, respectively, a supersolution of \eqref{linear scheme}, in the sense
\[	F^\D_j(x_i,\mathsf{c}_{i,j},[U_{i,j},U_{i,\bar{\jmath}}],U_{\cdot,j})\le 0 \quad\text{(resp., $\ge 0$)}\qquad \forall i\in\N, \,j=1,2,\,\bar{\jmath}=3-j,\]
then $\bU\le \bV$. 
\end{proposition}	
We omit the proof since it is very similar to \cref{prop:comparison}. \par
Now we describe the policy iteration method (Howard algorithm) for the fully discrete HJB equation \eqref{eq:HJBhD_1}. Given initial guess $(s_{\cdot,1}^{(0)},s_{\cdot,2}^{(0)})$, iterate for each $\iota \geq0$:

\noindent (i)  \textit{Policy evaluation.} Solve
\begin{equation}\label{eq:policy_ev}
     V_{i,j}^{\Delta, \i} = hu(c_{i,j}^{\i})+(1-\rho h)\left[\lambda_jhV_{i,\bar{\jmath}}^{\Delta,\i}+(1-\lambda_jh)\left(\sum_k\beta_k\left(x_i+hs^\i_{i,j}\right)V_{k,j}^{\Delta,\i}\right)\right].
\end{equation}\\
\noindent (ii) \textit{Policy update.} 

\begin{equation}\label{update_c}
\begin{aligned}
        &c^{(\iota+1)}_{i,j}=\mathop{\arg\max}\limits_{c\in \mC^{\D}_{j}(x_i)}\left\{
		hu(c)+(1-\rho h)(1-\lambda_jh)\left(\sum_k \beta_k(x_i+hs_{i,j}(c))V^{\D,\i}_{k,j}\right)\right\},\\
&s^{(\iota+1)}_{i,j}=s_{i,j}(c^{(\iota+1)}).		
\end{aligned}
\end{equation}
We now use the comparison principle in \cref{prop:comparison} and \cref{linear comparison} to show the global convergence of the Howard algorithm. This also gives the existence and uniqueness of a solution to \cref{Scheme}.
\begin{theorem}\label{Howard theo}
	Let $\bV^\i=(\bV^\i_1, \bV^\i_2)$, $\bV^\i_j\in  B(\mA^\Dx)$, be the sequence generated by the policy iteration method \eqref{eq:policy_ev}-\eqref{update_c}. Then
	\begin{equation}\label{eq:est_PI}
		\bV^\i \le \bV^{(\iota+1)} \qquad \forall \iota \in \N.
	\end{equation}
Moreover, $\lim_{\iota \to \infty}\bV^\i=\bV$ such that $\bV=(\bV_1,\bV_2)\in B(\mA^\Dx)^2$ is the unique solution of \cref{Scheme}.
\end{theorem}
\begin{proof}
The uniqueness of solution to \cref{Scheme} comes from \cref{prop:comparison}. We observe that from \cref{eq:policy_ev},
\begin{align*}
     V_{i,j}^{\Delta, \i} 
     \leq  \sup_{c\in \mC^{\D}_{j}}\left\{hu(c)+(1-\rho h)\left[\lambda_jhV_{i,\bar{\jmath}}^{\Delta,\i}+(1-\lambda_jh)\left(\sum_k\beta_k(x_i+hs_{i,j}(c))V_{k,j}^{\Delta,\i}\right)\right]\right\},
\end{align*}
hence $\bV^\i$ is a subsolution of \cref{Scheme}. It is clear that $(0,0)$ is a supersolution to \cref{Scheme}, hence  by \cref{prop:comparison} $\bV^\i\leq 0$ for all $\iota$.
	Moreover, with \cref{update_c}
	\begin{align*}
		V_{i,j}^{\Delta, \i} 
	\leq 	hu(c_{i,j}^{(\iota+1)})+(1-\rho h)\left[\lambda_jhV_{i,\bar{\jmath}}^{\Delta,\i}+(1-\lambda_jh)\left(\sum_k\beta_k\left(x_i+hs_{i,j}^{(\iota+1)}\right)V_{k,j}^{\Delta,\i}\right)\right].
	\end{align*}
Meanwhile
\begin{align*}
		V_{i,j}^{\Delta, (\iota+1)}=hu(c_{i,j}^{(\iota+1)})
		+(1-\rho h)\left[\lambda_jhV_{i,\bar{\jmath}}^{\Delta,(\iota+1)}+(1-\lambda_jh)\left(\sum_k\beta_k\left(x_i+hs_{i,j}^{(\iota+1)}\right)V_{k,j}^{\Delta,(\iota+1)}\right)\right].
	\end{align*}
	Therefore, $V^\i_{i,j}$ and  $V^{(\iota+1)}_{i,j}$ are sub and supersolution of the linear equation 
	$$
	F_j^\D(x_i,c_{i,j}^{(\iota+1)},[U_{i,j},U_{i,\bar{\jmath}}],U_{\cdot,j})=0,
	$$
it follows from \cref{linear comparison} that $\bV^\i \le \bV^{(\iota+1)}$. With $\bV^\i \le  0$ we conclude that $\bV^\i $ converges to some $\bV$. 
\end{proof}
	
We now consider additional properties of the scheme $\mathcal{F}^\D$.
\begin{proposition}\label{prop_scheme}
Assume that $\D x\sim h$ for $\D\to 0$. The scheme $\mathcal{F}^\D$, besides the monotonicity property established in \cref{monotone_scheme}, satisfies the following properties:
\begin{itemize}
\item[]{\em  Stability: } The unique  solution $\bV=(V^\D_{\cdot,1},V^\D_{\cdot,2})$ to \cref{Scheme} is uniformly bounded in $\D$.
	
	\item[]{\em  Consistency: } For $j=1,2$ and for any
	smooth function $\phi=(\phi_1,\phi_2)$
	\begin{align*}
			\limsup_{{\D\to 0}\atop{x_i\to x} } \mathcal{F}^\D_j(x_i,\phi(x_i),\phi_{j})  & \leq  \rho \phi_{j}(x)- H(x,y_j,D\phi_{j})-\lambda_j(\phi_{j}(x)-\phi_{\bar \jmath}(x))
		 \quad \forall x\in(\ux,\infty),\\
		\liminf_{{\D\to 0}\atop{x_i\to x}} \mathcal{F}^\D_j(x_i,\phi(x_i),\phi_{j})    &\geq \rho \phi_{j}(x)- H(x,y_j,D\phi_{j})-\lambda_j(\phi_{j}(x)-\phi_{\bar \jmath}(x)) 
	    \quad \forall x\in[\ux,\infty).
	\end{align*}
	\end{itemize}
\end{proposition}
\begin{proof}
The existence and uniqueness of solution to \cref{Scheme} has been considered in \cref{Howard theo}. We now use comparison to give the uniform bound. Clearly, $(0,0)$ is a supersolution to \cref{Scheme}. We now show the constant
\beq\label{def discrete sub}
\left(\frac{1}{\rho}\frac{(r\ux+y_1)^{1-\gamma}}{1-\gamma},\frac{1}{\rho}\frac{(r\ux+y_1)^{1-\gamma}}{1-\gamma}\right)
\eeq
 is a subsolution. Since \cref{def discrete sub} is a constant solution and from \cref{M pos}, we only need to show
\beq\label{discrete sub}
\frac{(r\ux+y_1)^{1-\gamma}}{1-\gamma}\leq \sup_{c\in \mC^{\D}_{j}(x_i)}\{u(c)\}\quad \forall x_i.
\eeq
From \cref{C set}, $\sup_{c\in \mC^{\D}_{j}(\ux)}\{u(c)\}=\frac{(r\ux+y_j)^{1-\gamma}}{1-\gamma}$ and $\sup_{c\in \mC^{\D}_{j}(x_i)}\{u(c)\}\geq \sup_{c\in \mC^{\D}_{j}(\ux)}\{u(c)\}$ if $x_i>\ux$ . We conclude \cref{discrete sub} by observing $y_1<y_2$.  \par
We now consider the consistency of scheme.  Given a  function $\phi=(\phi^1,\phi^2)$ with $\phi_{j}:\R\to\R$, we denote with 
$
	I[\phi_{j}](x):=\sum_{i\in  \N} \beta_i(x) \phi_{j} (x_i) 
$
its linear interpolation on the grid $\mA^\Dx$. If $\phi_{j}\in C^2(\R)$ with bounded derivative, then there exists a positive constant $C_1$ such that
\begin{equation}\label{eq:intepolation} 
	\sup_{x\in \R}| I[ \phi_{j}](x)-\phi_{j}(x)|\le C_1(\Dx)^2.
\end{equation} 
From \eqref{eq:intepolation} we obtain
\begin{equation}\label{eq:estimate_grad}
	\begin{split}
		& \left|\frac{1}{h} \big[ \sum_k \beta_k(x_i+hs_{i,j}(c))\phi_{j}(x_i)-\phi_{j}(x_i)]-D\phi_{j}(x_i)s_j(x_i)\right|\\
		=& \left|\frac{1}{h} \big[I[ \phi_{j}](x_i+hs_{i,j}(c))-\phi_{j}(x_i)\big]-D\phi_{j}(x_i)s_j(x_i)\right| \\
\le &\left|\frac{1}{h}  \big[\phi_{j}(x_i+hs_{i,j}(c))-\phi_{j}(x_i)\big]-D\phi_{j}(x_i)s_j(x_i)\right| +	C_1\frac{(\Dx)^2}{h}
\le C\left(\frac{(\Dx)^2}{h}+h\right),
	\end{split}
\end{equation}
we have from \cref{eq:estimate_grad}:
\begin{align*}
&\sup_{c\in \mC^{\D}_{j}(x_i)}\left\{
	u(c)+(1-\rho h)(1-\lambda_jh)D\phi_{j}(x_i)(rx_i+y_j-c)
	\right\}\\
\leq {}&\sup_{c\in \mC^{\D}_{j}(x_i)}\left\{
	u(c)+(1-\rho h)(1-\lambda_jh)\frac{1}{h} \left[ \sum_k \beta_k\left(x_i+hs_{i,j}(c)\right)\phi_{j}(x_i)-\phi_{j}(x_i)\right] 
	\right\}
	+C\left(\frac{(\Dx)^2}{h}+h\right).
\end{align*}
Let $x_i\to x\in (\ux,\infty)$ for $\D \to 0$, then $x_i>\ux$ when $\D$ is sufficiently small. From \cref{C_j}, we have
$
H(x_i, y_j,D\phi_{j})=\sup_{c\in \mC^{h}_{j}(x_i)}\left\{
	u(c)+(rx_i+y_j-c)D\phi_{j}(x_i)\right\} .
$
We then obtain
\begin{align*}
	  &\rho \phi_{j}(x_i)- H(x_i,y_j,D\phi_{j})-\lambda_j(\phi^{j}(x_i)-\phi_{\bar \jmath}(x_i))\\
	={}&\rho \phi_{j}(x_i)-\sup_{c\in \mC^{h}_{j}(x_i)}\left\{
	u(c)+(rx_i+y_j-c)D\phi_{j}(x_i)\right\} - \lambda_j(\phi_{\bar \jmath}(x_i)-\phi_{j}(x_i)) \\
	\ge{}& \rho \phi_{j}(x_i)-\sup_{c\in \mC^{h}_{j}(x_i)}\left\{
	u(c)+(1-\rho h)(1-\lambda_jh)D\phi_{j}(x_i)(rx_i+y_j-c)
	\right\}\\
	&-(1-\rho h)\lambda_j(\phi_{\bar \jmath}(x_i)-\phi_{j}(x_i))-C_1h\\
	\ge{}& \rho \phi_{j}(x_i)-(1-\rho h)\lambda_j(\phi_{\bar \jmath}(x_i)-\phi_{j}(x_i))-C\left(\frac{(\Dx)^2}{h}+h\right)\\
	&-\sup_{c\in \mC^{h}_{j}(x_i)}\left\{
	u(c)+\frac{(1-\rho h)(1-\lambda_jh)}{h} \left[ \sum_k \beta_k(x_i+hs_{i,j}(c))\phi_{j}(x_i)-\phi_{j}(x_i)\right] 
	\right\}
\end{align*}
and passing to the $\limsup$ in $(\ux,\infty)$,  we get the first condition in
{\em consistency}. 

%
Given $x\in [\ux,\infty)$, let $x_i\in \mA^\Dx$ such that $x_i\to x$. Then, taking into account \eqref{eq:estimate_grad}, we have
\begin{align*}
&\mathcal{F}^\D_j(x_i,\phi(x_i),\phi_{j})\\
\ge{}& \rho \phi_{j}(x_i)-\sup_{c\in \mC^{\D}_{j}(x_i)}\left\{
	u(c)+(1-\rho h)(1-\lambda_jh)D\phi_{j}(x_i)(rx_i+y_j-c)
	\right\}\\
	&-(1-\rho h)\lambda_j(\phi_{\bar \jmath}(x_i)-\phi_{j}(x_i))-C_1\left(h+\frac{(\Dx)^2}{h}\right)\\
\ge{}& \rho \phi_{j}(x_i)-\sup_{c\in (0,\infty)}\Big\{
	u(c)+(1-\rho h)(1-\lambda_jh)D\phi_{j}(x_i)(rx_i+y_j-c)
	\Big\}\\
	&-(1-\rho h)\lambda_j(\phi_{\bar \jmath}(x_i)-\phi_{j}(x_i))-C_1\left(h+\frac{(\Dx)^2}{h}\right)\\
={}& \rho \phi_{j}(x_i)- H(x_i, y_j,D\phi_{j})-\lambda_j(\phi_{j}(x_i)-\phi_{\bar \jmath}(x_i))-C\left(\frac{(\Dx)^2}{h}+h\right).
\end{align*}
Passing to the $\liminf_{{\D\to 0}\atop{x_i\to x}}$ in the previous inequality, we get the second condition in {\em consistency}.

\end{proof}

\begin{proposition}
Assume $\D\sim h$ and  $\D$ is sufficiently small, then the solution $V^\D_{i,j}$ to the scheme \eqref{Scheme} is nondecreasing in $i$. 
\end{proposition}
\begin{proof}
We aim to show 
\beq\label{Vi+1}
\mathcal{F}^\D_j\left(x_i,\left[V^{\Delta}_{i+1,j},V^{\Delta}_{i+1,\bar \jmath}\right],V^{\Delta}_{\cdot,j}\right)\geq 0\qquad i\in\NN,\,j=1,2,\,\bar{\jmath}=3-j,
\eeq
and then apply the discrete comparison principle. For \cref{Vi+1} to hold, we only need
\beq\label{sup i+1}
\begin{aligned}
&\sup_{c\in \mC^{\D}_{j}(x_{i+1})}\left\{
 u(c)+(1-\rho h)(1-\lambda_jh)\frac{1}{h}\left(\sum_k \beta_k(x_{i+1}+hs_{i+1,j}(c))V_{k,j} \right)\right\}\\
 \geq {}&\sup_{c\in \mC^{\D}_{j}(x_{i})}\left\{
 u(c)+(1-\rho h)(1-\lambda_jh)\frac{1}{h}\left(\sum_k \beta_k(x_i+hs_{i,1}(c))V_{k,j} \right)\right\},
 \end{aligned}
\eeq
where the $\sup$ on the right hand side of inequality \eqref{sup i+1} is attained by  $c^*_{i,j}$. Notice that since $c^*_{i,j}\in \mC^{\D}_{j}(x_{i})$ we have $c^*_{i,j}+r\D x+\frac{\D x}{h}\in \mC^{\D}_{j}(x_{i+1}).$
From $\D\sim h$ and  $\D$ being sufficiently small, we have $r\D x+\frac{\D x}{h}>0$ for any fixed $r$. Since $c^*_{i,j}+r\D x+\frac{\D x}{h}$ is an admissible but possibly suboptimal control at $(x_{i+1},y_j)$, we can obtain
$$
\begin{aligned}
&\sup_{c\in \mC^{\D}_{j}(x_{i+1})}\left\{
 u(c)+(1-\rho h)(1-\lambda_jh)\frac{1}{h}\left(\sum_k \beta_k(x_{i+1}+hs_{i+1,j}(c))V_{k,j} \right)\right\}\\
 \geq {}&
 u\left(c^*_{i,j}+r\D x+\frac{\D x}{h}\right)
+(1-\rho h)(1-\lambda_jh)\frac{1}{h}\left(\sum_k \beta_k\left(x_{i+1}+hs_{i+1,j}\left(c^*_{i,j}+r\D x+\frac{\D x}{h}\right)\right)V_{k,j} \right)\\
 \geq {}&
 u(c^*_{i,j})+(1-\rho h)(1-\lambda_jh)\frac{1}{h}\left(\sum_k  \beta_k(x_{i}+hs_{i,j}^*)V_{k,j} \right).
 \end{aligned}
$$
In the second inequality, we used monotonicity of utility $u(\cdot)$ and:
\begin{align*}
& x_{i+1}+hs_{i+1,j}\left(c^*_{i,j}+r\D x+\frac{\D x}{h}\right)\\
={}&x_i+\D x+h\left(r(x_i+\D x)+y_j-\left(c^*_{i,j}+r\D x+\frac{\D x}{h}\right)\right)=x_i+hs_{i,j}^*.
\end{align*}
We then have \cref{sup i+1} and therefore \cref{Vi+1}.
\end{proof}
We give the main convergence result. Recall that $\bV\in B(\mA^\Dx)^2$ is the solution to the discrete HJB equation \eqref{Scheme}. We define the numerical solution $v^\D(x):=I[\bV](x)$ by using the interpolation operator \eqref{Eq:intrep}.
\begin{theorem}\label{thm:convergence}
Let $v$ be the unique viscosity solution to the system \eqref{eq:HJB}. Then, $v^\D(x)\to v(x)$ locally uniformly on $[\ux,+\infty)$ as $\D\to 0$. 
\end{theorem}
\begin{proof}
We define the bounded functions
$
\bar{\su}(x):=\limsup_{\substack{z\to x\\ \D \to 0}}v^\D(z),\quad \underline{\sv}(x):=\liminf_{\substack{z\to x\\ \D \to 0}}v^\D(z).
$
By definition we have $\bar{\su}(x)\geq \underline{\sv}(x)$. By using properties of the scheme (monotonicity, stability and consistency) it is standard to show that $\bar{\su}$ and $\underline{\sv}$ are respectively a sub and supersolution to the system \eqref{eq:HJB}. Using the comparison principle, we obtain $\bar{\su}(x)\leq \underline{\sv}(x)$. Therefore $v(x)=\bar{\su}(x)= \underline{\sv}(x)$ is the viscosity solution. 
\end{proof}

\section{Numerical analysis}

\subsection{Approximation of the policies under state constraints}
In the theoretical analysis, $c^*_{i,j}$ is defined as the $\arg\max$ in \cref{eq:fully_discrete_scheme} ($i$). In practice, $c^*_{i,j}$ can be obtained in at least three different approaches:
\begin{itemize}
\item Solve an optimization problem on each grid $(i,j)$ using $\sf{fminbound}$ function in python, this is most consistent with analysis but least efficient;
\item Use a discretized control space and then an $\arg\max$ function, there is a trade-off between efficiency and accuracy when choosing the control space;
\item Use the first order condition from the system \eqref{eq:HJB} $c_j^*(x)=(Dv_j(x))^{-1/\gamma}$, this is the most effective one in practice.
\end{itemize}
In the third approach, we use the finite difference derivative ${\mathbf{D}}(V_{i,j}^{\Delta})$ to approximate $Dv_j(x_i)$, where
\beq
{\mathbf{D}}(V_{i,j}^{\Delta})=
\left\{\begin{aligned}
\qquad &\frac{V_{1,j}^{\Delta}-V_{0,j}^{\Delta}}{\D x},\qquad &&i=0,\\
&\frac{V_{i-1,j}^{\Delta}-2V_{i,j}^{\Delta}+V_{i+1,j}^{\Delta}}{2\D x},&&0<i<N_x,\\
&\frac{V_{N_x,j}^{\Delta}-V_{N_x-1,j}^{\Delta}}{\D x},&&i=N_x.
\end{aligned}\right.
\eeq
To enforce the state constraint at $\ux$, we use \cref{Eq:C_j} and policy update step
\beq\label{approx pi}
c^{(\iota+1)}_{i,j} = \min \left\{\left[{\mathbf{D}}(V_{i,j}^{\Delta,\i})\right]^{-\tfrac{1}{\gamma}},\frac{i\D x-\ux}{h}+r(\ux+i\D x)+y_j\right\}.
\eeq

\subsection{Algorithms and implementation}

\begin{algorithm}[H]
    \caption{Howard algorithm with fixed $r$}\label{Howard}
    \KwData{Initial values $\bc_j^{(0)}$, and parameters $r,\lambda_{1},\lambda_{2}, h, \D x, y_1,y_2$.}
    \KwResult{Solution $\bV$}
    \Do{$\max_j\|\bc^{(\iota+1)}_j-\bc^{(\iota)}_j\|_{l^\infty}\geq 10^{-5}$}{
    \textit{Policy evaluation}: Solve for $\bV^\i$
    \begin{equation*}
    \begin{aligned}
       & \begin{bmatrix}
            (h\rho-1)(1-\lambda_1h)M(\bfs_1^\i)+I & \lambda_1h(h\rho-1)I \\
            \lambda_2h(h\rho-1)I & (h\rho-1)(1-\lambda_2h)M(\bfs_2^\i)+I
        \end{bmatrix}
        \begin{bmatrix}
            \bV_1^\i\\
            \bV_2^\i
        \end{bmatrix} \\
        ={}& h\begin{bmatrix}
            u(\bc_1^\i)\\
            u(\bc_2^\i)
        \end{bmatrix}
    \end{aligned}
    \end{equation*}
    \\ 
    Calculate 
    $
    \bc^{(\iota+1)}_j
    $
    with \cref{update_c} or \cref{approx pi}
    \Comment{Consumption policy update}\\
    $
    \bfs_j^{(\iota+1)}=s_{\cdot,j}^{(\iota+1)},\quad s_{i,j}^{(\iota+1)}=rx_i+y_j-c_{i,j}^{(\iota+1)}.
    $\Comment{Saving policy update}
        }
   \end{algorithm}
In all the algorithms in this section, $I$ stand for a $N\times N$ identity matrix. We first introduce the Howard algorithm \ref{Howard} for solving the HJB equation with a fixed $r$. At each iteration $\iota$, we start with two vectors $\bc_j^\i$ and construct the $N\times N$ matrix $M(\bfs_j^\i)$. The \textit{Policy evaluation} step is the vector form of \cref{eq:policy_ev}. In particular, $M(\bfs_j^\i)$ is tridiagonal if $h$ is chosen such that $x_i+h(rx_i+y_j-c_{i,j})\in [x_{i-1},x_{i+1}]\quad \forall i.$

\begin{algorithm}[H]
    \caption{Stationary Aiyagari model}\label{Aiyagari algo}
    \KwData{Initial values $c^{(0)}, r^{(0)}$, and parameters $\lambda_{1},\lambda_{2}, h, \D x, y_1,y_2, \alpha, \delta$.}
    \KwResult{Solution $\bV, \bG$, optimal policies $\bc,\mathbf{s}$, equilibrium $r$}
    \Do{$\|r^{(\tau+1)}-r^{(\tau)}\|\geq 10^{-5}$}{
        Solve HJB equation using Howard algorithm \ref{Howard} with $r=r^{(\tau)}$ and output optimal consumption policy as $\bc_j^{(\tau)}$, $j=1,2$ \\
     $
    \bfs^{(\tau)}_j=s^{(\tau)}_{\cdot,j},\quad s^{(\tau)}_{i,j}=r^{(\tau)}x_i+y_j-c_{i,j}^{(\tau)}
    $\Comment{Update saving policy}\\
    Solve for $\bG^{(\tau)}=(\bG^{(\tau)}_1,\bG^{(\tau)}_2)$: $\sum_k G^{(\tau)}_{k,1}+\sum_k G^{(\tau)}_{k,2}=1/\D x$,
     $$
       \begin{bmatrix}
            (1-\lambda_1h)M^{\texttt{T}}(\bfs^{(\tau)}_1)-I & \lambda_2hI\\
            \lambda_1hI & (1-\lambda_2h)M^{\texttt{T}}(\bfs^{(\tau)}_2)-I
        \end{bmatrix}
        \begin{bmatrix}
            \bG^{(\tau)}_1\\
            \bG^{(\tau)}_2
        \end{bmatrix}=0
   $$ \Comment{Find the invariant distribution}\\
    $
     \sum_k x_k G^{(\tau)}_{k,1}\D x+\sum_k x_k G^{(\tau)}_{k,2}\D x=K^{(\tau)}
    $\Comment{Update aggregate asset}\\
    $
     r^{(\tau+1)}= A\alpha \left(\frac{N}{K^{(\tau)}}\right)^{1 - \alpha}-\delta
    $\Comment{Update interest rate}
}
\end{algorithm}

\begin{algorithm}[H]
    \caption{Dynamic Aiyagari model}\label{algo:D Aiyagari}
    \KwData{Initial values $r^{(0)}_n, \bV^{(0),N}_{j}, \bG^{(0),0}_{j}$, and parameters $\lambda_1, \lambda_2,\rho, h, y_1,y_2, \alpha, \delta$.}
    \KwResult{$\bV^{n}, \bG^{n}$, optimal policies $\bc^n,\mathbf{s}^n$, equilibrium $r_n$}
    \Do{$\max_n\|r_n^{(\tau+1)}-r_n^{(\tau)}\| \geq 10^{-4}$}{
    \For{$n=N-1$ \KwTo $0$}{
           $
                c^{(\tau),n}_{i,j} =\min \left\{ [{\mathbf{D}}(V^{(\tau),n+1}_{i,j})]^{-1/\gamma},\frac{i\D x-\ux}{h}+r^{(\tau)}_n(\ux+i\D x)\right\}, \,\bfs^{(\tau),n}_{j}=r^{(\tau)}_n\mathbf{x}+y_j-\bc^{(\tau),n}_{j}
            $
       \begin{equation*}
    \begin{aligned}
       &\begin{bmatrix}
            \bV_{1}^{(\tau),n}\\
            \bV_{2}^{(\tau),n}
        \end{bmatrix}= h\begin{bmatrix}
            u(\bc^{(\tau),n}_{1})\\
            u(\bc^{(\tau),n}_{2})
        \end{bmatrix}\\
        &+\begin{bmatrix}
            (1-h\rho)(1-\lambda_1h)M(\bfs_{1}^{(\tau),n})  & \lambda_1h(1-h\rho)I \\
            \lambda_2h(1-h\rho)I & (1-h\rho)(1-\lambda_2h)M(\bfs_{1}^{(\tau),n}) 
        \end{bmatrix}
        \begin{bmatrix}
            \bV_{1}^{(\tau),n+1}\\
            \bV_{2}^{(\tau),n+1}
        \end{bmatrix}
    \end{aligned}
    \end{equation*}
\Comment{Solve the HJB equation backward in time}
        }
    \For{$n=0$ \KwTo $n-1$}{
    $$
           \begin{bmatrix}
                \bG^{(\tau),n+1}_{1}\\
                \bG^{(\tau),n+1}_{2}
            \end{bmatrix}
            =
            \begin{bmatrix}
                (1-\lambda_1h)M^{\texttt{T}}(\bfs_{1}^{(\tau),n}) & \lambda_2hI \\
                \lambda_1hI & (1-\lambda_2h)M^{\texttt{T}}(\bfs_{2}^{(\tau),n})
            \end{bmatrix}
            \begin{bmatrix}
                \bG^{(\tau),n}_{1}\\
                \bG^{(\tau),n}_{2}
            \end{bmatrix}
    $$\Comment{Update the probability distribution forward in time}
        }
    $
    \sum_k x_k G^{(\tau),n}_{k,1}+\sum_k x_k G^{(\tau),n}_{k,2}=K^{(\tau)}_n
    $
    \Comment{Update aggregate asset}
    \\
    $
    r_n^{(\tau+1)}= A\alpha \left(N_n^{(\tau)}\right)^{1 - \alpha}\left(K_n^{(\tau)}\right)^{\alpha-1}-\delta
    $\Comment{Update interest rate}
    }
\end{algorithm}

In Algorithm \ref{Aiyagari algo} we use Algorithm \ref{Howard} each time for solving the HJB equation with an updated interest rate $r^{(\tau)}$. At iteration $\tau$, $\bc^{(\tau)}_j$ is the output of Algorithm \ref{Howard} and at the end of the inner iteration loop the matrix $M\left(\bfs^{(\tau)}_j\right)$, $j=1,2$ are stored. The transposition $M^{\texttt{T}}\left(\bfs^{(\tau)}_j\right)$ is immediately used for solving the FPK equation. After solving for $\bG^{(\tau)}$ we update the aggregate asset $K^{(\tau)}$ and then $r^{(\tau)}$. To solve the Huggett model with Algorithm \ref{Aiyagari algo}, we use the bi-section method for modifying the interest rate update step: if $K^{(\tau)}-B> 10^{-5}$ then $r^{(\tau+1)} =\frac{r_{min}+r^{(\tau)}}{2}$, else if $K^{(\tau)}-B< 10^{-5}$ then $r^{(\tau+1)} = \frac{r_{max}+r^{(\tau)}}{2}$ until $\vert K^{(\tau)}-B\vert \geq 10^{-5}$. \par
For solving dynamic Aiyagari model with Algorithm \ref{algo:D Aiyagari}, at each iteration $\tau$ and with a fixed flow of interest rate $r^{(\tau)}_n$ we use backward induction for solving the evolutive HJB equation and then forward time marching for the FPK equation. Analogously to the stationary model, at iteration $\tau$ the matrix $M\left(\bfs_{j}^{(\tau),n}\right)$ for time index $n$ is the transposition $M^{\texttt{T}}\left(\bfs_{j}^{(\tau),n}\right)$, which is used for solving the evolutive FPK equation. \\

\subsection{Examples}
\noindent \textbf{Stationary models}\\
We first consider the stationary Aiyagari model with the parameters:
$
\ux=-0.15, y_1=0.1, y_2=0.5, \lambda_1=\lambda_2=0.4, \alpha=0.35, \delta=0.1.
$
We compare results for different risk aversion $\gamma$. The results with $\gamma=2$ shows consistency with the plots in \cite[Numerical Appendix]{achdou2022income}. The results with $\gamma=4$ show that higher risk aversion leads to less consumption (more saving), less concentration at $\ux$ and lower interest rate. \cref{c prime} shows the asymptotic behavior $c'(x)\rightarrow r+\frac{1}{\gamma}(\rho-r)$ as $x\rightarrow +\infty$. 
\begin{figure}[h!]
	\centering
	\caption{ Value function with $\gamma=2$ (left) and probability distribution (right)}\label{cs Aiyagari}
	\begin{tabular}{cc}
	\includegraphics[width=0.46\textwidth]{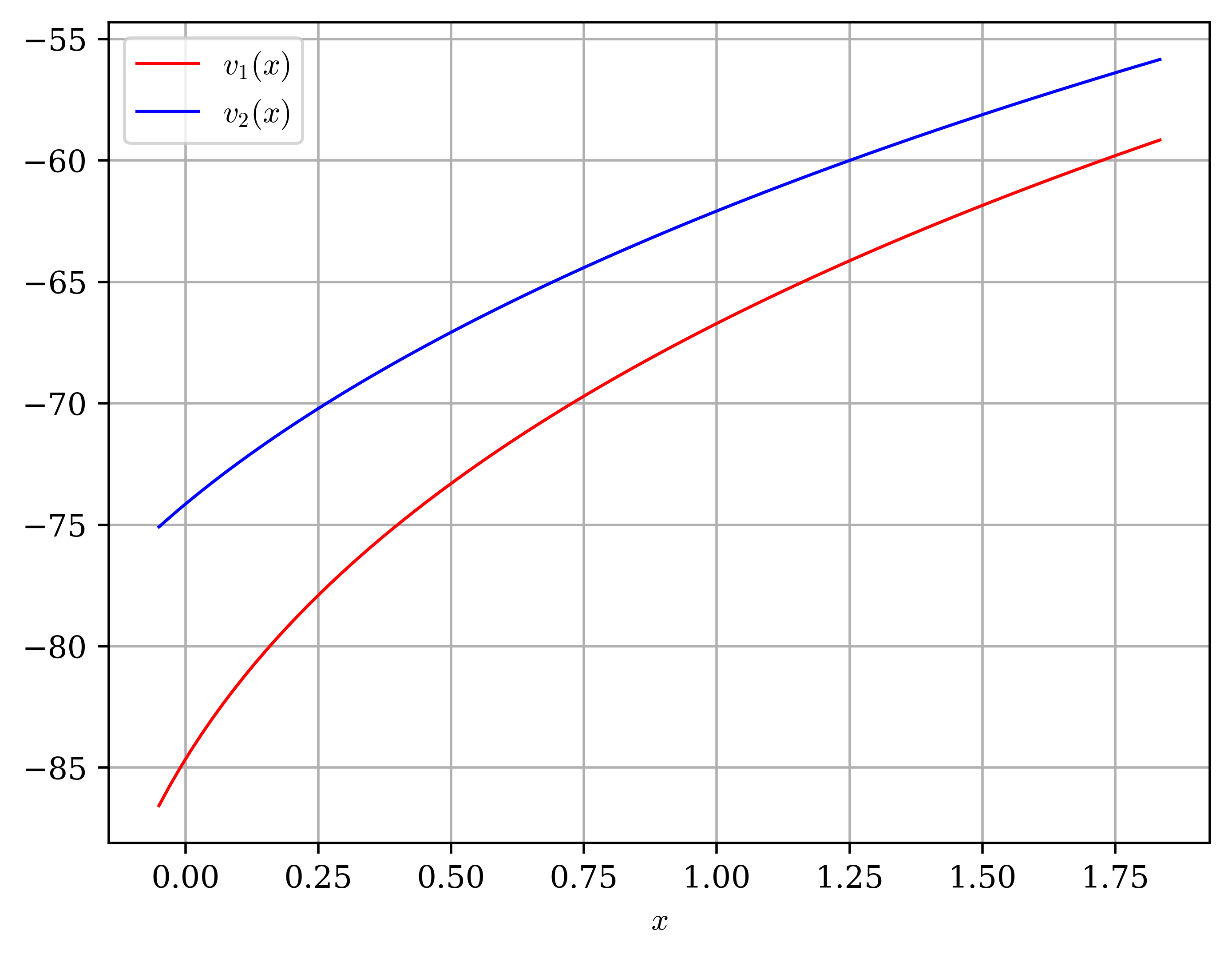} &
	\includegraphics[width=0.46\textwidth]{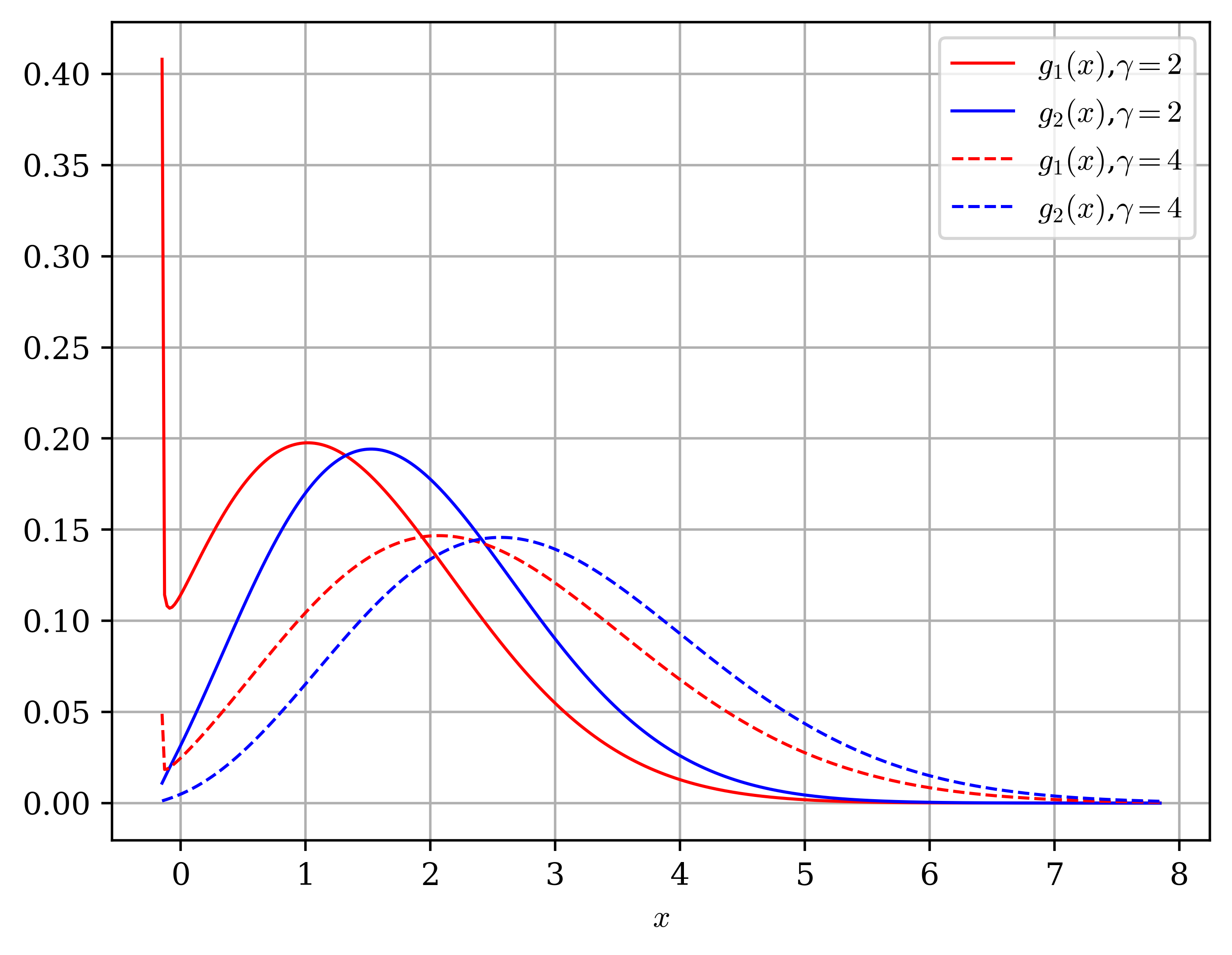} \\
	\end{tabular}
\end{figure}
\begin{figure}[ht]
	\centering
	\caption{ Consumption (left) and saving (right)}\label{cs Aiyagari}
	\begin{tabular}{cc}
	\includegraphics[width=0.46\textwidth]{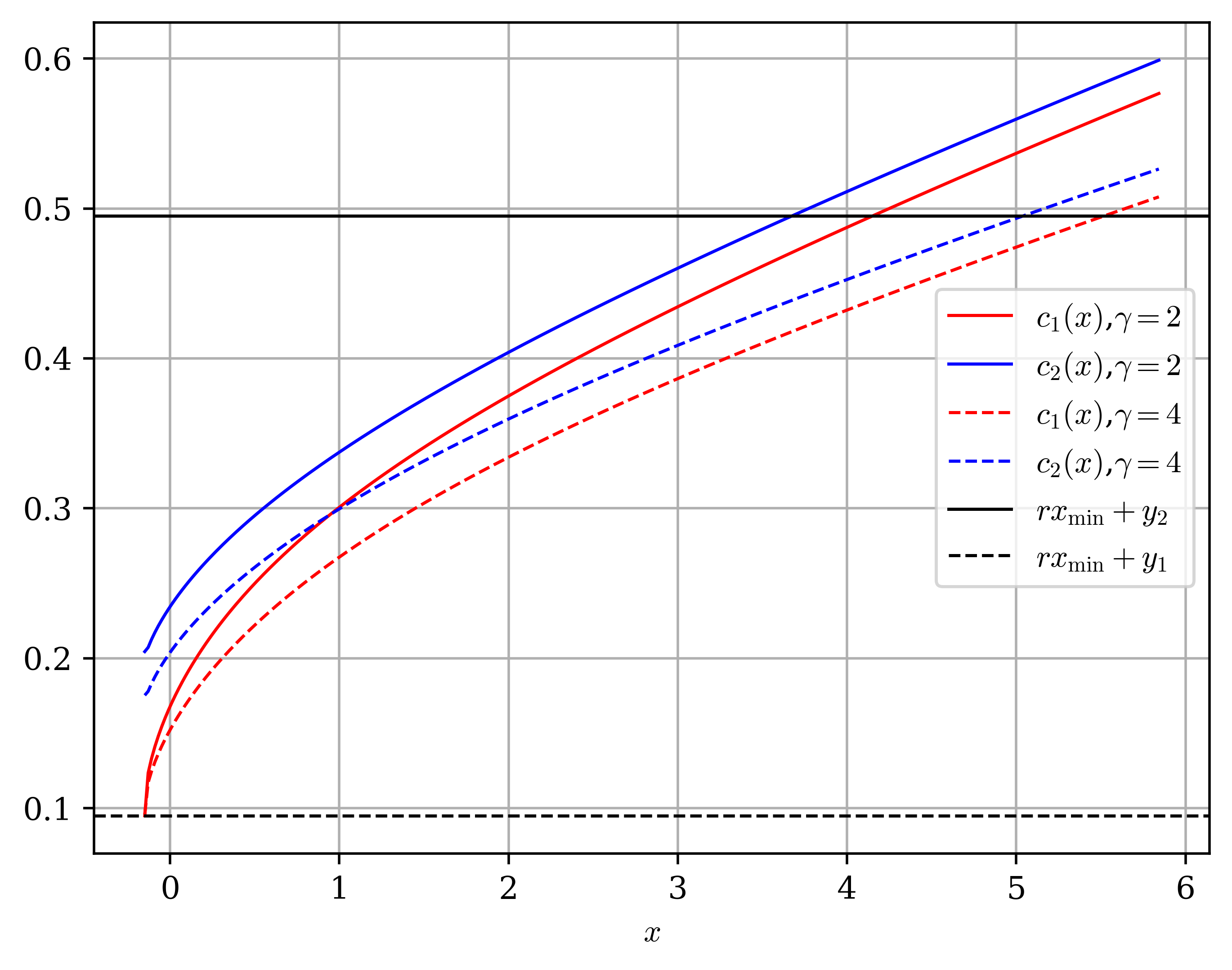} &
	\includegraphics[width=0.46\textwidth]{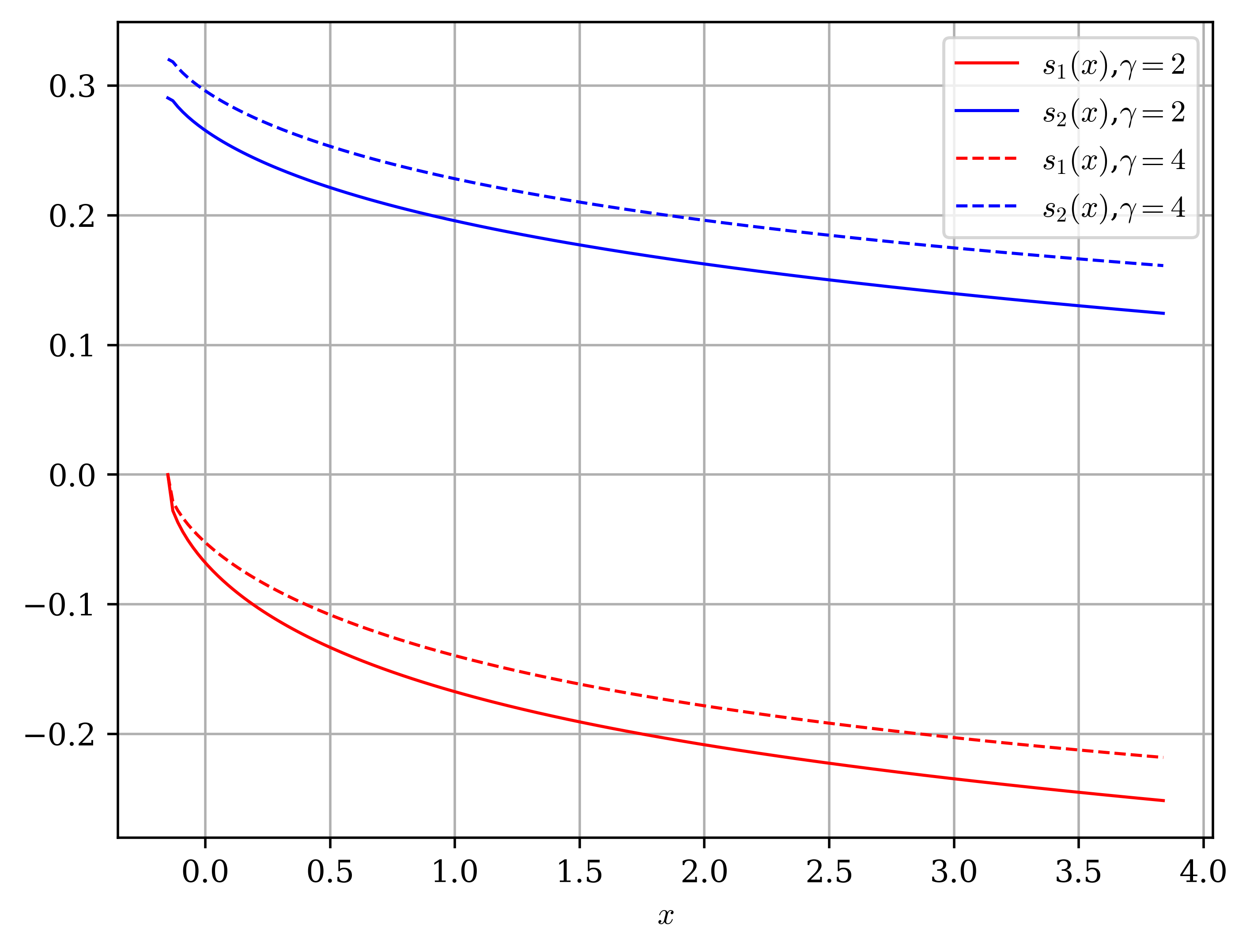} \\
	\end{tabular}
\end{figure}
\begin{figure}[ht]
   \centering
    \caption{MPC for $\gamma=2$ and $\gamma=4$}\label{c prime}
    \includegraphics[width=0.45\linewidth]{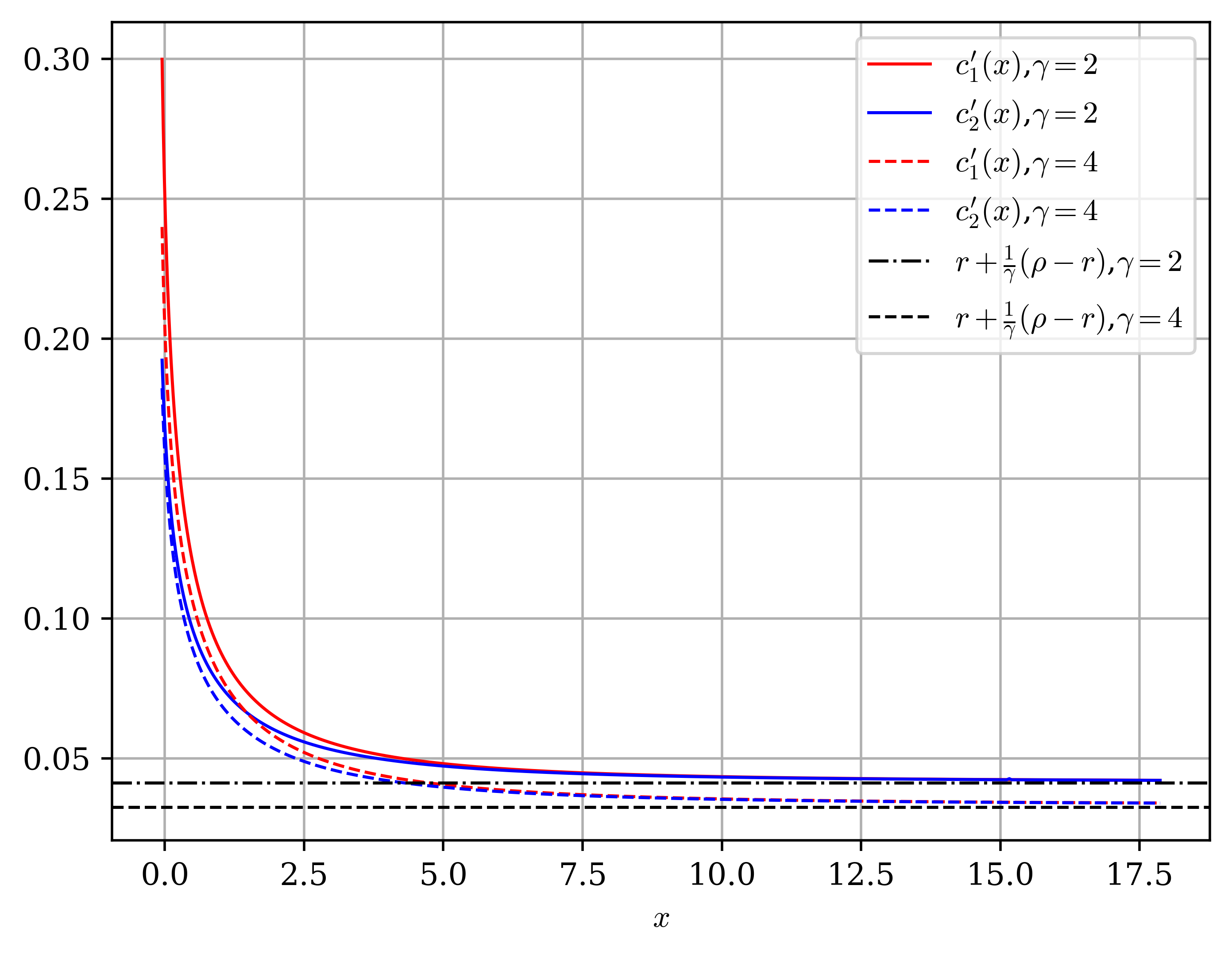}
\end{figure}\\
\noindent \textbf{Transition models}\\
To study the transition to this stationary equilibrium, we set $A(t) = A^{st}$ and initialize the system with a distribution $dm_j(0)$ obtained from a stationary model with a different initial productivity level $A(0)$. A sudden shift in productivity—commonly referred to as an MIT shock—induces time evolution in both the distribution $dm(t)$ and the interest rate $r(t)$ as the economy converges toward the new steady state.\par
 Considering a model starting with TFP $A=0.9$ and wealth distribution $\mathsf{g}_j$. We use a stationary equilibrium with $A^{{st}}=1.0$ to define the terminal value function. We then solve the dynamic Aiyagari model on the time horizon $[0,T]$ with Algorithm \ref{algo:D Aiyagari}. \cref{transition interest} shows the result $r(t)$ and we observe the interest rate fluctuation in response to this shock. After an initial jump of $r(t)$, $K(t)$ will increase so that $r(t)$ goes down until the stationary state $r^{{st}}$.
\begin{figure}[h]
    \centering
    \caption{Evolution of interest rate in a dynamic Aiyagari model }\label{transition interest}
    \includegraphics[width=0.45\linewidth]{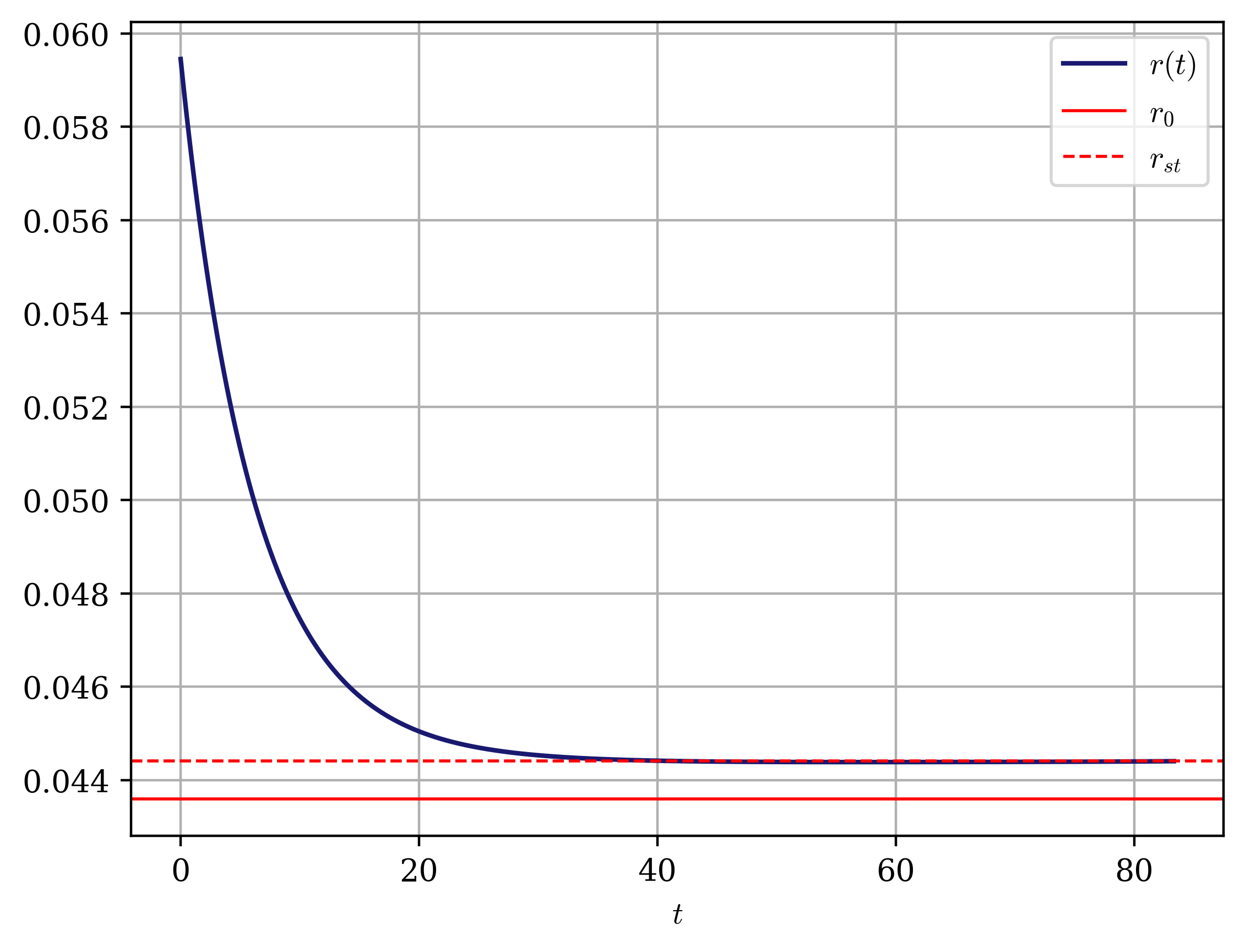}
\end{figure}

\bigskip
\noindent {\bf Acknowledgment.}
\small{This work is completed while Qing Tang is a visitor in Universit\'e Paris Cit\'e, supported by China Scholarship Council No. 202406410221. Qing Tang thanks Prof. Yves Achdou for discussion and hospitality. Yong-Shen Zhou acknowledges the support of the China Scholarship Council No. 202404610016.}

\bibliographystyle{siam} 
\bibliography{Aiyagari_SL_arxiv}

%
%
%

%
	\end{document}